\documentclass[12pt]{amsart}
\usepackage{amssymb,latexsym}
\usepackage{pdfsync}
\usepackage{color}

\usepackage{esint}

\newdimen\AAdi%
\newbox\AAbo%
\def\AAk#1#2{\s_etbox\AAbo=\hbox{#2}\AAdi=\wd\AAbo\kern#1\AAdi{}}%
\def\AAr#1#2#3{\s_etbox\AAbo=\hbox{#2}\AAdi=\ht\AAbo\raise#1\AAdi\hbox{#3}}%
\font\tenmsb=msbm10 at 12pt
\font\sevenmsb=msbm7 at 8pt
\font\fivemsb=msbm5 at 6pt
\newfam\msbfam
\textfont\msbfam=\tenmsb
\scriptfont\msbfam=\sevenmsb
\scriptscriptfont\msbfam=\fivemsb
\def\Bbb#1{{\tenmsb\fam\msbfam#1}}

\newcommand{\beq}{\begin{equation}}
\newcommand{\eeq}{\end{equation}}
\newcommand{\beqr}{\begin{eqnarray}}
\newcommand{\eeqr}{\end{eqnarray}}
\newcommand{\ba}{\begin{array}}
\newcommand{\ea}{\end{array}}

\begin{document}

\newtheorem{theorem}{Theorem}[section]
\newtheorem{lemma}{Lemma}[section]
\newtheorem{corollary}{Corollary}[section]
\newtheorem{remark}{Remark}
\newtheorem{proposition}{Proposition}[section]
\newtheorem{definition}{Definition}
\newtheorem{eg}{Example}
\newtheorem*{claim}{Claim}
\newcommand{\noi}{\noindent}
\newcommand{\dis}{\displaystyle}
\newcommand{\mint}{-\!\!\!\!\!\!\int}
\numberwithin{equation}{section}

\def \bx{\hspace{2.5mm}\rule{2.5mm}{2.5mm}}
\def \vs{\vspace*{0.2cm}}
\def\hs{\hspace*{0.6cm}}
\def \ds{\displaystyle}
\def \p{\partial}
\def \O{\Omega}
\def \o{\omega}
\def \b{\beta}
\def \m{\mu}
\def \l{\lambda}
\def\L{\Lambda}
\def \ul{u_\lambda}
\def \D{\Delta}
\def \d{\delta}
\def \k{\kappa}
\def \s{\sigma}
\def \e{\varepsilon}
\def \a{\alpha}
\def \sm{\sigma}
\def \tf{\tilde{f}}
\def\cqfd{%
\mbox{ }%
\nolinebreak%
\hfill%
\rule{2mm} {2mm}%
\medbreak%
\par%
}
\def \pr {\noindent {\it Proof.} }
\def \rmk {\noindent {\it Remark} }
\def \esp {\hspace{4mm}}
\def \dsp {\hspace{2mm}}
\def \ssp {\hspace{1mm}}

\def\la{\langle}\def\ra{\rangle}

\def \u{u_+^{p^*}}
\def \ui{(u_+)^{p^*+1}}
\def \ul{(u^k)_+^{p^*}}
\def \energy{\int_{\R^n}\u }
\def \sk{\s_k}
\def \mo{\mu_k}
\def\cal{\mathcal}
\def \I{{\cal I}}
\def \J{{\cal J}}
\def \K{{\cal K}}
\def \OM{\overline{M}}

\def\n{\nabla}

\def\fk{{{\cal F}}_k}
\def\M1{{{\cal M}}_1}
\def\Fk{{\cal F}_k}
\def\Fl{{\cal F}_l}
\def\FF{\cal F}
\def\Gk{{\Gamma_k^+}}
\def\n{\nabla}
\def\uuu{{\n ^2 u+du\otimes du-\frac {|\n u|^2} 2 g_0+S_{g_0}}}
\def\uuug{{\n ^2 u+du\otimes du-\frac {|\n u|^2} 2 g+S_{g}}}
\def\sku{\sk\left(\uuu\right)}
\def\qed{\cqfd}
\def\vvv{{\frac{\n ^2 v} v -\frac {|\n v|^2} {2v^2} g_0+S_{g_0}}}
\def\vvs{{\frac{\n ^2 \tilde v} {\tilde v}
 -\frac {|\n \tilde v|^2} {2\tilde v^2} g_{S^n}+S_{g_{S^n}}}}
\def\skv{\sk\left(\vvv\right)}
\def\tr{\hbox{tr}}
\def\pO{\partial \Omega}
\def\dist{\hbox{dist}}
\def\RR{\Bbb R}\def\R{\Bbb R}
\def\C{\Bbb C}
\def\B{\Bbb B}
\def\N{\Bbb N}
\def\Q{\Bbb Q}
\def\Z{\Bbb Z}
\def\PP{\Bbb P}
\def\EE{\Bbb E}
\def\F{\Bbb F}
\def\G{\Bbb G}
\def\H{\Bbb H}
\def\SS{\Bbb S}\def\S{\Bbb S}

\def\div{\hbox{div}\,}

\def\lcf{{locally conformally flat} }

\def\circledwedge{\setbox0=\hbox{$\bigcirc$}\relax \mathbin {\hbox
to0pt{\raise.5pt\hbox to\wd0{\hfil $\wedge$\hfil}\hss}\box0 }}

\def\sss{\frac{\s_2}{\s_1}}

\date{\today}

\title[ Hamilton's gradient estimates ]{Hamilton's gradient estimates and Liouville theorems for $u_{t}=\Delta u^{m}+au\log u+bu$ on Riemannian manifolds}

\author{}

 \author[Jun Sun]{Jun Sun } 
\address{School of Mathematics and Statistics\\ Wuhan University\\Wuhan 430072,
China
 }
 \email{sunjun@whu.edu.cn}

\author[Jiaming Yang]{Jiaming Yang}
\address{School of Mathematics and Statistics\\ Wuhan University\\Wuhan 430072,
China
 }
 \email{2020302142030@whu.edu.cn}

\begin{abstract}
    In this paper,  we apply Nash-Moser iteration and Saloff-Coste's Sobolev inequalities to derive gradient estimates for positive solutions to a class of nonlinear parabolic equations of the form
    \begin{equation*}
    u_{t}=\Delta u^{m}+a(x, t)u\log u+b(x, t)u, \quad 1-\frac{2}{n}<m<1 \quad \text{or}\quad  1<m<1+\frac{1}{1+\sqrt{2n}}, 
    \end{equation*}
on a complete Riemannian manifold $(M,g)$ of dimension $n$, where $a(x, t)$, $b(x, t)$ are $C^{1}$ functions. For $a=b=0$, this equation becomes the porous medium equation (PME) or the fast diffusion equation (FDE) when $m>1$ or $m<1$, respectively. As consequences, we obtain Liouville type theorems for the corresponding constant coefficient equations when
    \begin{equation*}
    1-\frac{1}{\sqrt{n-1}}<m<1\quad \text{and}\quad 1<m<1+\frac{1}{\sqrt{n-1}}. 
    \end{equation*}

\vskip12pt

\noindent{\it Keywords and phrases}:gradient estimate; Nash-Moser iteration; Liouville type theorem.

\noindent {\it MSC 2020}: 58J05, 35B45.

\end{abstract}
\maketitle
\raggedbottom
\setlength{\jot}{5pt}
\setlength{\abovedisplayskip}{7pt plus 1pt minus 1pt}
\setlength{\belowdisplayskip}{7pt plus 1pt minus 1pt}
\setlength{\abovedisplayshortskip}{3pt plus 1pt}
\setlength{\belowdisplayshortskip}{5pt plus 1pt minus 1pt}
\section{Introduction}

\allowdisplaybreaks[4]

\vspace{.1in}

The well-known Liouville theorem says that any bounded harmonic function in Euclidean space $\mathbb{R}^{n}$ is constant. In 1975, Yau \cite{Yau} proved that any positive harmonic function on a Riemannian manifold with non-negative Ricci curvature is constant. Since then, many researchers began to pay attention to some nonlinear equations related to Laplace equation. Later in 1986, Li and Yau \cite{LY} derived the celebrated Li-Yau estimate for positive solutions to the heat equation:
\begin{theorem}\label{thm;Li-Yau} (\cite{LY})
Let $(M, g)$ be a compact Riemannian manifold with $\mathrm{Ric}\,_{M}\geqslant -k$ for some $k\geqslant 0$. Let $B_{2R}\subset M$ be the geodesic ball of radius $2R$ centered at a fixed point $O$. Suppose $u$ is any positive solution to heat equation $u_{t}=\Delta u$ on $B_{2R}\times [0, \infty)$. Then for $\alpha >1$, there holds in $B_{R}$ that
\begin{equation*}
    \sup_{B_{R}}\left(\frac{|\nabla u|{^2}}{u{^2}}-\alpha \frac{u_{t}}{u}\right)\leqslant \frac{C\alpha {^2}}{R{^2}}\left(\frac{\alpha {^2}}{\alpha -1}+\sqrt{k}R\right)+\frac{n\alpha {^2}k}{2(\alpha -1)}+\frac{n\alpha {^2}}{2t}. 
\end{equation*}
\end{theorem}
This work facilitated the widespread application of gradient estimates. In 1993, Hamilton \cite{Ha} derived the following gradient estimate for closed manifolds:
\begin{theorem}[\cite{Ha}]\label{thm;Hamilton}
Let $(M, g)$ be a compact Riemannian manifold of dimension $n\geqslant 2$ with $\mathrm{Ric}\,_{M}\geqslant -k$ for some $k\geqslant 0$. Suppose $u$ is any positive solution to the heat equation
\begin{equation*}
\frac{\partial u}{\partial t}=\Delta u
\end{equation*}
with $u\leqslant A$. Then 
\begin{equation*}
\frac{\left|\nabla u\right|^{2}}{u^{2}}\leqslant \left( \frac{1}{t}+2k \right) \log \frac{A}{u}. 
\end{equation*}
\end{theorem}
This result can compare the derivatives at different positions at the same time, and was extended by Souplet and Zhang \cite{SZ} to complete noncompact manifolds in 2006. Except for the classical harmonic equation and heat equation, the gradient estimates for solutions to many other equations deserve to be studied. 

Let $(M, g)$ be a Riemannian manifold of dimension $n\geqslant 2$ and $m>0$ be a parameter. The equation
\begin{equation}\label{equ;0}
\frac{\partial u}{\partial t}=\Delta u^{m},
\end{equation}
which arises in physics, is a generalization of the heat equation ($m=1$). It is called porous medium equation (PME for short) when $m>1$ and fast diffusion equation (FDE for short) when $m<1$, respectively.

The mathematical theory of PME and FDE is based on a priori estimate. In 1979, Aronson and B\'{e}nilan \cite{AB} obtained a celebrated second-order differential inequality of the form
\begin{equation*}
\sum\limits_{i}\frac{\partial }{\partial x^{i}}\left( mu ^{m-2}\frac{\partial u}{\partial x^{i}} \right) \geqslant -\frac{\kappa }{t}, \quad \kappa :=\frac{n}{n(m-1)+2}, 
\end{equation*}
which applies to all positive smooth solutions of (\ref{equ;0}) defined on the whole Euclidean space on the condition that $\displaystyle{m>1-\frac{2}{n}}$. In 2009, Lu, Ni, V\'{a}zquez and Villani \cite{L} studied PME and FDE on manifolds and derived a local Li-Yau type gradient estimate. 

Many studies focused on Hamilton type estimates of solutions to PME and FDE. In \cite{Z2} and \cite{Z1}, Zhu obtained Hamilton type gradient estimate and Liouville type theorem:
\begin{theorem}[\cite{Z2}, \cite{Z1}]\label{thm;Zhu}
Let $(M, g)$ be a Riemannian manifold of dimension $n\geqslant 2$ with $\mathrm{Ric}\,_{M}\geqslant -k$ for some $k\geqslant 0$. Suppose $u$ is any positive solution to (\ref{equ;0}) in $Q_{R, T}=B_{R}(x_{0})\times [t_{0}-T, t_{0}]\subset M\times (-\infty, \infty)$. Then:
\begin{flushleft}
(i) If $\displaystyle{1-\frac{2}{n}}<m<1$, there exists a positive constant $C=C(m, n)$ such that
\begin{equation*}
\frac{\left|\nabla v\right|}{v^{1/2}}\leqslant CM^{1/2}\left( \frac{1}{R}+\frac{1}{\sqrt{T}}+\sqrt{k} \right) 
\end{equation*}
in $Q_{R/2, T/2}$, where $v=m u^{m}/(1-m)$ and $M=\sup_{Q_{R, T}}v$;

(ii) If $\displaystyle{1<m<1+\frac{1}{1+\sqrt{2n}}}$, there exists a positive constant $C=C(m, n)$ such that
\begin{equation*}
v^{\frac{1}{4}\frac{2-m}{m-1}}\left|\nabla v\right|\leqslant CM^{1+\frac{1}{4}\frac{2-m}{m-1}}\left( \frac{1}{R}+\frac{1}{\sqrt{T}}+\sqrt{k} \right) 
\end{equation*}
in $Q_{R/2, T/2}$, where $v=m u^{m}/(m-1)$ and $M=\sup_{Q_{R, T}}v$. 
\end{flushleft}
\end{theorem}
This result was proved by using the maximum principle, and was improved by Xu \cite{X}, Huang and Ma \cite{HM}, Wang, Xie and Zhang \cite{WXZ}, and Huang, Xu and Zeng \cite{HXZ}. Commonly, these works apply the maximum principle to give the estimates, and the arguments differ in the calculation of the coefficient of the highest-order term, which we hope to be positive. This leads to the difference in the valid range of $m$.

Recently in 2025, Huang and Shen \cite{HS} derived Li-Yau type gradient estimates for PME and FDE in the weak sense via Nash-Moser iteration. They improved the results of Lu, Ni, V\'{a}zquez and Villani in \cite{L} and Huang, Huang and Li in \cite{HHL}. 

In the other direction, in 2024, Wang and Wang \cite{WW2,WW1} applied Nash-Moser iteration to prove the gradient estimates for solutions to a class of nonlinear elliptic equations of the form
\begin{equation}\label{equ;57}
\Delta u+a(x)u\log u+b(x)u=0. 
\end{equation}
The conditions of their results contain the integral Ricci curvature conditions and integral coefficient conditions. This type of equation (\ref{equ;57}) is closely related to that of Euler-Lagrange equations of the $\mathcal{W}$-entropy and Log-Sobolev functional on $(M, g)$. Similar ideas apply to various nonlinear equation (\cite{HHW}, \cite{JWZ}, etc.).

In 2018, Zhang and Zhu \cite{ZZ} derived Li-Yau type gradient estimates of the solutions to heat equation under integral Ricci curvature conditions. Later in 2020, Wang \cite{W} solved a similar problem for solutions to $u_{t}=\Delta u+au\log u$. Their method depends on the properties of heat kernel. Unfortunately, although researchers have proved many properties of the fundamental solutions to PME and FDE on Euclidean space and some special Riemannian manifolds (see \cite{GMP} and \cite{V} for instance), there are few related results on general manifolds.

\vspace{.1in}

In this paper we are concerned with a class of nonlinear parabolic equations 
\begin{equation}\label{equ;1}
\frac{\partial u}{\partial t}=\Delta u^{m}+a(x,t)u\log u+b(x,t)u, 
\end{equation}
where $a$, $b$ are $C^{1}$ functions on $M\times \mathbb{R}$ and $m>0$. We call (\ref{equ;1}) \textbf{PME type} if $m>1$ and \textbf{FDE type} if $0<m<1$. Different from Huang and Shen's result of Li-Yau type gradient estimates, we apply Nash-Moser iteration to study the Hamilton's gradient estimates and related issues for the positive solutions to (\ref{equ;1}). We also extend the method in \cite{HXZ} to FDE to improve the valid range of $m$. In fact, the gradient estimates hold when $1-\frac{1}{\sqrt{n-1}}<m<1$ and $a$, $b$ are constants.

Throughout this paper, we fix a point $(x_{0}, t_{0})\in M\times \mathbb{R}$ and denote $B_{R}=B_{R}(x_{0})$, where $B_{R}(x_{0})$ is the geodesic ball of radius $R$ centered at $x_{0}$. Our first main result is the following Hamilton type gradient estimate for positive solutions to PME type equation:
\begin{theorem}\label{thm;R1}
    Let $(M, g)$ be a Riemannian manifold of dimension $n\geqslant 2$ with $\mathrm{Ric}\,_{M}\geqslant -k$ for some $k\geqslant 0$. Suppose that $u$ is any positive solution to (\ref{equ;1}) on $Q_{R}=B_{R}\times (t_{0}-R^{2}, t_{0}]$,  where $1<m<1+\frac{1}{1+\sqrt{2n}}$. Suppose also that $u$ is bounded from below and above by positive constants. Assume that for $0<R\leqslant 1$, $a$, $b$, $\left|\nabla a\right|$ and $\left|\nabla b\right|$ are bounded functions on $Q_{R}$, and denote
    \begin{equation*}
    \sup_{Q_{R}}\left|a\right|=D_{1}\quad ,\quad \sup_{Q_{R}}\left|b\right|=D_{2}\quad ,\quad \sup_{Q_{R}}\left|\nabla a\right|=D_{3}\quad ,\quad \sup_{Q_{R}}\left|\nabla b\right|=D_{4}. 
    \end{equation*}
    Let $v=m u^{m-1}/(m-1)$, then we have for $0<R\leqslant 1$
    \begin{equation}\label{equ;58}
    \sup_{(x, t)\in Q_{R/2}}\left( v^{\frac{2-m}{4(m-1)}}\left|\nabla v\right|+v^{\frac{2-m}{4(m-1)}} \right) \leqslant C\max\left\{ v_{\min}^{-1/2}, 1 \right\} \frac{1+\sqrt{k}R}{R}, 
    \end{equation}
    where $C=C(m, n, D_{1}, D_{2}, D_{3}, D_{4}, v_{\max})$ is a constant. 
\end{theorem}

As an application, when $a, b$ are constants, we can expand the valid range of $m$ and remove the restriction $0<R\leqslant 1$ (see Theorem \ref{thm;2}). Also, we do not need $u$ to be bounded from below when $a\geqslant 0$. In particular, we can derive a Liouville type theorem. 
\begin{corollary}\label{cor;1}
Let $(M, g)$ be a noncompact Riemannian manifold of dimension $n\geqslant 2$ without boundary and with non-negative Ricci curvature. Suppose that $u$ is a positive solution to (\ref{equ;1})  on $M\times (-\infty,T_0]$, where $1<m<1+\frac{1}{\sqrt{n-1}}$ and $a$, $b$ are constants. Suppose also that $u$ is bounded from  above. Then:
\begin{flushleft}
(i) If $a>0$, then $u$ cannot be bounded from above by $e^{-\frac{2}{2m-1}}e^{-\frac{b}{a}}$. That is, 
\begin{equation*}
\sup_{(x, t)\in M\times (-\infty,T_0]}u(x, t)> e^{-\frac{2}{2m-1}}e^{-\frac{b}{a}};
\end{equation*} 

(ii) If $a=0$ and $b<0$, then such $u$ does not exist;

(iii) If $a=b=0$, then $u$ must be a positive constant;

(iv) If $a<0$ and $u\geqslant e^{-\frac{2}{2m-1}}e^{-\frac{b}{a}}$ on $M\times (-\infty,T_0]$, then $u$ must be  $e^{-\frac{b}{a}}$ identically. 
\end{flushleft}
\end{corollary}

For FDE type equations, we can prove parallel results of Hamilton type gradient estimates:

\begin{theorem}\label{thm;R2}
Let $(M, g)$ be a Riemannian manifold of dimension $n\geqslant 2$ with $\mathrm{Ric}\,_{M}\geqslant -k$ for some $k\geqslant 0$. Suppose that $u$ is any positive solution to (\ref{equ;1}) on $Q_{R}=B_{R}\times (t_{0}-R^{2}, t_{0}]$,  where $1-\frac{2}{n}<m<1$. Suppose also that $u$ is bounded from below and above by positive constants. Assume that for $0<R\leqslant 1$, $a$, $b$, $\left|\nabla a\right|$ and $\left|\nabla b\right|$ are bounded functions on $Q_{R}$, and denote
\begin{equation*}
\sup_{Q_{R}}\left|a\right|=D_{1}\quad ,\quad \sup_{Q_{R}}\left|b\right|=D_{2}\quad ,\quad \sup_{Q_{R}}\left|\nabla a\right|=D_{3}\quad ,\quad \sup_{Q_{R}}\left|\nabla b\right|=D_{4}. 
\end{equation*}
Let $v=m u^{m-1}/(1-m)$, then we have
\begin{equation}\label{equ;62}
\sup_{(x, t)\in Q_{R/2}}\left( v^{-\frac{2-m}{4(1-m)}}\left|\nabla v\right|+v^{-\frac{2-m}{4(1-m)}} \right) \leqslant C \frac{1+\sqrt{k}R}{R}, 
\end{equation}
where $C=C(m, n, D_{1}, D_{2}, D_{3}, D_{4}, v_{\max}, v_{\min})$ is a constant. 
\end{theorem}

We can also derive the following Liouville type result for FDE:

\begin{corollary}\label{cor;R4}
Let $(M,g)$ be a noncompact Riemannian manifold of dimension $n\geqslant2$ without boundary and with non-negative Ricci curvature. Suppose that $u$ is a positive smooth solution to (\ref{equ;1}) on $M\times\mathbb{R}$, where $1-\frac{1}{\sqrt{n-1}}<m<1$ and $a,b$ are constants. Suppose also that $u$ is bounded from above. In parts (1), (2) and (4), assume in addition that $u$ is bounded from below by a positive constant when
\begin{equation*}
 2\leqslant n\leqslant4,\qquad
 1-\frac1{\sqrt{n-1}}<m\leqslant\frac{n-1}{n+3}.
\end{equation*}
Then:
\begin{flushleft}
(1) When $a=b=0$, $u$ must be a positive constant;

(2) When $a=0$ and $b\neq0$, if one of the following conditions holds,
\begin{equation*}
\begin{aligned}
 \frac{n-1}{n+3}<m<1&,\quad (\text{this is trivial if }n=5),\\
 1-\frac1{\sqrt{n-1}}<m\leqslant\frac{n-1}{n+3}&,\quad 2\leqslant n\leqslant4,\quad b>0,\\
 1-\frac1{\sqrt{n-1}}<m\leqslant\frac{n-1}{n+3}&,\quad n>5,\quad b<0,
\end{aligned}
\end{equation*}
then such $u$ does not exist;

(3) When $a>0$,

$\quad$(i) If $2\leqslant n\leqslant4$ and $m<1/2$, then $u$ cannot be bounded from below by $e^{-2/(2m-1)}e^{-b/a}$. That is,
\begin{equation*}
 \inf_{(x,t)\in M\times\mathbb{R}}u(x,t)
 <e^{-\frac2{2m-1}}e^{-\frac ba};
\end{equation*}

$\quad$(ii) If $2\leqslant n\leqslant4$, $m>1/2$, or $n\geqslant5$, then $u$ cannot be bounded from above by $e^{-2/(2m-1)}e^{-b/a}$. That is,
\begin{equation*}
 \sup_{(x,t)\in M\times\mathbb{R}}u(x,t)
 >e^{-\frac2{2m-1}}e^{-\frac ba}.
\end{equation*}
No independent positive lower bound for $u$ is required in part (3).

(4) When $a<0$, if one of the following conditions holds,
\begin{alignat*}{3}
2\leqslant &n\leqslant4,&\hspace{10pt}\quad \frac{n-1}{n+3}&<m<1,&\hspace{10pt} \\*
2\leqslant &n\leqslant4,&\hspace{10pt}\quad 1-\frac1{\sqrt{n-1}}&<m\leqslant\frac{n-1}{n+3},
 &\hspace{10pt}\quad u\leqslant e^{-\frac2{2m-1}}e^{-\frac ba}, \\
&n=5,&\hspace{10pt}\quad \frac12&<m<1,&\hspace{10pt} \\
&n>5,&\hspace{10pt}\quad \frac{n-1}{n+3}&<m<1,&\hspace{10pt} \\*
&n>5,&\hspace{10pt}\quad 1-\frac1{\sqrt{n-1}}&<m\leqslant\frac{n-1}{n+3},
 &\hspace{10pt}\quad u\geqslant e^{-\frac2{2m-1}}e^{-\frac ba},
\end{alignat*}
then $u$ is spatially constant and there exists a constant $C_0\geqslant0$ such that
\begin{equation*}
 u(x,t)=\exp\left(-\frac ba-C_0e^{at}\right)
 \qquad\text{on }M\times\mathbb{R}.
\end{equation*}
Moreover, $C_0=0$ if and only if $u$ is bounded from below by a positive constant. In particular, $u\equiv e^{-b/a}$ in the second and fifth cases above.
\end{flushleft}
\end{corollary}

We study these equations by integral methods. The PME estimates and the variable-coefficient FDE estimates use Nash--Moser iteration and Saloff--Coste's Sobolev inequalities. For the remaining constant-coefficient FDE estimates, we localize the gradient quantity before testing and use the resulting highest-order absorption to pass to the high-power limit.

\vspace{.1in}

The subsequent sections are organized as follows: In Section \ref{PME}, we will prove the Hamilton-type gradient estimates and the corollaries for PME; In Section \ref{FDE}, we will prove the corresponding results for FDE.

\vspace{.2in}

\section{PME Type Equations}\label{PME}

\vspace{.1in}

Let $v=mu ^{m-1}/(m-1)$, then direct computation yields
\begin{equation}\label{equ;2}
v_{t}-(m-1)v\Delta v=\left|\nabla v\right|^{2}+av\log v+b'v, 
\end{equation}
where
\begin{equation*}
b'=(m-1)b-a\log \frac{m}{m-1}. 
\end{equation*}
We introduce the differential operator from \cite{L}
\begin{equation*}
\mathcal{L}:=\frac{\partial }{\partial t}-(m-1)v\Delta . 
\end{equation*}
As preparation, we record some formulas of $\mathcal{L}$.
\begin{lemma}\label{lemma;1}
Let $u$ be a positive solution to (\ref{equ;1}) on Riemannian manifold $(M, g)$ of dimension $n$ for some $m>1$. Let $v=mu ^{m-1}/(m-1)$ and $\beta $ be any constant ($\beta \neq 1$). Then we have
\begin{equation}\label{equ;3}
\mathcal{L}(v)=\left|\nabla v\right|^{2}+av\log v+b'v, 
\end{equation}
\begin{equation}\label{equ;4}
\mathcal{L}(v^{-\beta })=-\beta v^{-\beta -1}(1+(m-1)(\beta +1))\left|\nabla v\right|^{2}-\beta v^{-\beta }(a\log v+b'), 
\end{equation}
\begin{equation}\label{equ;5}
\begin{aligned}
\mathcal{L}(\left|\nabla v\right|^{2})&=-2(m-1)v\left|\nabla ^{2}v\right|^{2}-2(m-1)v\mathrm{Ric}\,(\nabla v, \nabla v)+2(m-1)\Delta v \left|\nabla v\right|^{2}\\
&\quad +2\left<\nabla v, \nabla \left|\nabla v\right|^{2} \right>+2(a\log v+b')\left|\nabla v\right|^{2}+2v\log v\left<\nabla v, \nabla a \right>\\
&\quad +2a \left|\nabla v\right|^{2}+2v\left<\nabla v, \nabla b' \right>.
\end{aligned}
\end{equation}

\begin{proof}
    We only prove (\ref{equ;5}). By Bochner formula we have
    \begin{align*}
&\quad \mathcal{L}(\left|\nabla v\right|^{2}) \\*
&\quad =\frac{\partial }{\partial t}\left|\nabla v\right|^{2}-(m-1)v\Delta \left|\nabla v\right|^{2} \\
&\quad =2\left< \nabla v, \nabla v_{t}\right>-2(m-1)v\left|\nabla ^{2}v\right|^{2}-2(m-1)v\left<\nabla v, \nabla (\Delta v) \right>-2(m-1)v\mathrm{Ric}\,(\nabla v, \nabla v) \\
&\quad =2\left<\nabla v, \nabla \mathcal{L}(v) \right>+2(m-1)\Delta v\left|\nabla v\right|^{2}-2(m-1)v\left|\nabla ^{2}v\right|^{2}-2(m-1)v\mathrm{Ric}\,(\nabla v, \nabla v) \\
&\quad =-2(m-1)v\left|\nabla ^{2}v\right|^{2}-2(m-1)v\mathrm{Ric}\,(\nabla v, \nabla v)+2(m-1)\Delta v\left|\nabla v\right|^{2} \\
&\quad \quad +2\left<\nabla v, \nabla \left|\nabla v\right|^{2} \right>+2(a\log v+b')\left|\nabla v\right|^{2}+2v\log v\left<\nabla v, \nabla a \right> \\*
&\quad \quad +2a \left|\nabla v\right|^{2}+2v\left<\nabla v, \nabla b' \right>.
\end{align*}
\end{proof}
\end{lemma}

We define 
\begin{equation}\label{equ;67}
    n'=\begin{cases} n, & n\geqslant 3, \\
    4 , & n=2.
    \end{cases}
\end{equation}
\begin{lemma}\label{lemma;2}
Let $(M, g)$ be a complete Riemannian manifold of dimension $n$, $Q_{R}=B_{R}(x_{0})\times (t_{0}-R^{2}, t_{0}]\subset M\times \mathbb{R}$. Suppose that $\varphi =\varphi (x, t)$ is a function which satisfies
\begin{equation}\label{equ;61}
X \sup_{t_{0}-\rho ^{2}<t\leqslant t_{0}}\int_{B_{\rho }}\varphi ^{2}(x, t)+Y\int_{t_{0}-\rho ^{2}}^{t_{0}}\left( \int_{B_{\rho }}\varphi ^{\frac{2n'}{n'-2}} \right)^{\frac{n'-2}{n'}}  d t\leqslant S, 
\end{equation}
where $B_{\rho }=B_{\rho }(x_{0})$ and $X, Y, S, \rho $ are constants, $0<\rho \leqslant R$. Then
\begin{equation*}
\int_{Q_{\rho }}\varphi ^{\frac{2(n'+2)}{n'}}\leqslant \frac{S^{(n'+2)/n'}}{X^{2/n'}Y}, 
\end{equation*}
where $Q_{\rho }=B_{\rho }(x_{0})\times (t_{0}-\rho ^{2}, t_{0}]\subset Q_{R}$. Equivalently, 
\begin{equation*}
\left\|\varphi \right\|_{L^{\frac{2(n'+2)}{n'}}(Q_{\rho })}^{2}\leqslant X^{-\frac{2}{n'+2}}Y^{-\frac{n'}{n'+2}}S. 
\end{equation*}

\begin{proof}
    We assume that $n\geqslant 3$. The following argument is also vaild as long as $n$ is replaced by $n'=4$ when $n=2$.
    
    For $t_{0}-\rho ^{2}<t\leqslant t_{0}$, we use H\"{o}lder's inequality to deduce that
    \begin{equation*}
    \begin{aligned}
    \int_{B_{\rho }}\varphi ^{\frac{2(n+2)}{n}}&=\int_{B_{\rho }}\varphi ^{2}\cdot \varphi ^{\frac{4}{n}}\leqslant \left( \int_{B_{\rho }}\varphi ^{\frac{2n}{n-2}} \right) ^{\frac{n-2}{n}}\left( \int_{B_{\rho }}\varphi ^{2} \right) ^{\frac{2}{n}}\\
    &\leqslant \sup_{t_{0}-\rho ^{2}<t\leqslant t_{0}}\left( \int_{B_{\rho }}\varphi ^{2} \right) ^{\frac{2}{n}}\left( \int_{B_{\rho }}\varphi ^{\frac{2n}{n-2}} \right) ^{\frac{n-2}{n}}. 
    \end{aligned}
    \end{equation*}
    Integrating over $t_{0}-\rho ^{2}<t\leqslant t_{0}$ yields
    \begin{equation*}
    \int_{Q_{\rho }}\varphi ^{\frac{2(n+2)}{n}}\leqslant \sup_{t_{0}-\rho ^{2}<t\leqslant t_{0}}\left( \int_{B_{\rho }}\varphi ^{2} \right) ^{\frac{2}{n}}\int_{t_{0}-\rho ^{2}}^{t_{0}}\left( \int_{B_{\rho }}\varphi ^{\frac{2n}{n-2}} \right) ^{\frac{n-2}{n}}  d t. 
    \end{equation*}
    Applying the inequality in the assumption (\ref{equ;61}), we obtain
    \begin{equation*}
    \int_{Q_{\rho }}\varphi ^{\frac{2(n+2)}{n}}\leqslant \left( \frac{S}{X} \right) ^{\frac{2}{n}}\frac{S}{Y}= \frac{S^{(n+2)/n}}{X^{2/n}Y}, 
    \end{equation*}
    which completes the proof. 
\end{proof}
\end{lemma}

We need the following Sobolev inequality:
\begin{lemma}[Saloff-Coste, \cite{S-C}]\label{lemma;3}
Let $(M, g)$ be a Riemannian manifold of dimension $n\geqslant 2$ with $\mathrm{Ric}\,_{M}\geqslant -k$ for some $k\geqslant 0$. If $n>2$, there exists a positive constant $C_{n}$ depending only on $n$, such that for any $\varphi \in C_{0}^{\infty}(B_{R})$, where $B_{R}\subset M$ is a geodesic ball, there holds
\begin{equation}\label{equ;sobolev}
e^{-C_{n}(1+\sqrt{k}R)}\left|B_{R}\right|^{2/n}R^{-2}\left\|\varphi \right\|_{L^{\frac{2n}{n-2}}(B_{R})}^{2}\leqslant \int_{B_{R}}\left|\nabla \varphi \right|^{2}+R^{-2}\int_{B_{R}}\varphi ^{2}. 
\end{equation}
If $n=2$, the inequality (\ref{equ;sobolev}) holds with $n$ replaced by  $n'$. 
\end{lemma}

Now we can prove Theorem \ref{thm;R1}. 

\begin{proof}[proof of Theorem \ref{thm;R1}]
    Throughout the Sobolev--Moser part of the proof, we set
\begin{equation*}
     n'=n\quad(n\geqslant3),\qquad n'=4\quad(n=2).
\end{equation*}
Thus the spatial Sobolev exponent, the parabolic exponent and the Moser amplification factor are $2n'/(n'-2)$, $2(n'+2)/n'$ and $(n'+2)/n'$, respectively.  The actual geometric dimension remains $n$ in the Bochner formula, the Hessian inequality and all algebraic coefficients. We only prove the theorem for $n\geqslant 3$. The argument is also valid when $n=2$ as long as all the $n$ in these exponents are replaced by $n'$.

    We write $b'$ as $b$ for simplicity. For any $\displaystyle{\frac{R}{2}\leqslant \rho \leqslant R}$ and $t_{0}-\rho ^{2}<\tau \leqslant t_{0}$, we denote
    \begin{equation*}
    Q_{\rho }=B_{\rho }\times (t_{0}-\rho ^{2}, t_{0}] \quad \text{and}\quad Q_{\rho , \tau }=B_{\rho }\times (t_{0}-\rho ^{2}, \tau ]. 
    \end{equation*}
    Let
    \begin{equation*}
    f=\left|\nabla v\right|^{2}+1 \quad \text{and}\quad F=v^{-\beta }f=\frac{\left|\nabla v\right|^{2}+1}{v^{\beta }}, 
    \end{equation*}
    where $\beta $ is a (negative) constant to be determined later. Applying Lemma \ref{lemma;1} we have
    \begin{align*}
\mathcal{L}(F)&=f\mathcal{L}(v^{-\beta })+v^{-\beta }\mathcal{L}(f)-2(m-1)v\left<\nabla v^{-\beta },\nabla f \right> \\*
&=-\beta v^{-\beta -1}(1+(m-1)(\beta +1))f(f-1)-\beta v^{-\beta }(a\log v+b)f \\
&\quad -2(m-1)v^{-\beta +1}\left|\nabla ^{2}v\right|^{2}-2(m-1)v^{-\beta +1}\mathrm{Ric}\,(\nabla v, \nabla v) \\
&\quad +2(m-1)v^{-\beta }(f-1)\Delta v+2v^{-\beta }\left<\nabla v, \nabla f \right>+2(a\log v+b)v^{-\beta }(f-1) \\
&\quad +2v^{-\beta +1}\log v\left<\nabla v, \nabla a \right>+2av^{-\beta }(f-1)+2v^{-\beta +1}\left<\nabla v, \nabla b \right> \\*
&\quad +2(m-1)\beta v^{-\beta }\left<\nabla v, \nabla f \right>.
\end{align*}
    Note that $\displaystyle{\left|\nabla ^{2}v\right|^{2}\geqslant \frac{(\Delta v)^{2}}{n}}$ and $\mathrm{Ric}\,(\nabla v, \nabla v)\geqslant -k\left|\nabla v\right|^{2}$. One obtains
    \begin{align}
-\mathcal{L}(F)&\geqslant \frac{2(m-1)}{n}v^{-\beta +1}(\Delta v)^{2}-2(m-1)v^{-\beta }(f-1)\Delta v-2(m-1)kv^{-\beta +1}(f-1) \label{equ;7} \\*
&\quad +\beta (1+(m-1)(\beta +1))v^{-\beta -1}f^{2}-2(1+\beta (m-1))v^{-\beta }\left<\nabla v, \nabla f \right> \notag \\
&\quad -(2-\beta )(a\log v+b)v^{-\beta }f-2v^{-\beta +1}\log v\left<\nabla v, \nabla a \right>-2av^{-\beta }f+2av^{-\beta } \notag \\*
&\quad -2v^{-\beta +1}\left<\nabla v, \nabla b \right>-\beta (1+(m-1)(\beta +1))v^{-\beta -1}f+2(a\log v+b)v^{-\beta }. \notag
\end{align}
    By Cauchy-Schwarz inequality, 
    \begin{equation}\label{equ;8}
    \frac{2(m-1)}{n}v^{-\beta +1}(\Delta v)^{2}-2(m-1)v^{-\beta }(f-1)\Delta v\geqslant -\frac{n(m-1)}{2}v^{-\beta -1}(f-1)^{2}, 
    \end{equation}
    \begin{equation}\label{equ;9}
    \begin{aligned}
    -2v^{-\beta +1}\log v\left<\nabla v, \nabla a \right>&\geqslant -2v^{-\beta +1}\left|\log v\right|\left|\nabla v\right|\left|\nabla a\right|\\
    &\geqslant -\max\left\{ e^{-1},v_{\max}\left|\log v_{\max}\right| \right\} v^{-\beta }(\left|\nabla v\right|^{2}+\left|\nabla a\right|^{2})\\
    &\geqslant -C_{1}v^{-\beta }(f+\left|\nabla a\right|^{2}), 
    \end{aligned}
    \end{equation}
    \begin{equation}\label{equ;10}
    -2v^{-\beta +1}\left<\nabla v, \nabla b \right>\geqslant -v^{-\beta +1}(\left|\nabla v\right|^{2}+\left|\nabla b\right|^{2})\geqslant -v^{-\beta +1}f-\left|\nabla b\right|^{2}v^{-\beta +1}.
    \end{equation}
    Moreover, since $f=v^{\beta }F$, we have
    \begin{equation}\label{equ;11}
    v^{-\beta }\left<\nabla v, \nabla f \right>=\beta v^{-\beta -1}f^{2}-\beta v^{-\beta -1}f+\left<\nabla F, \nabla v \right>.
    \end{equation}
    Substituting (\ref{equ;8})$\sim $(\ref{equ;11}) into $(\ref{equ;7})$ gives
    \begin{align*}
-\mathcal{L}(F)&\geqslant \left\{ -(m-1)\beta ^{2}+(m-2)\beta -\frac{n(m-1)}{2} \right\} v^{-\beta -1}f^{2}+n(m-1)v^{-\beta -1}f \\*
&\quad -\frac{n(m-1)}{2}v^{-\beta -1}-2(1+\beta (m-1))\left<\nabla F, \nabla v \right>-2(m-1)kv^{-\beta +1}f \\
&\quad +2(m-1)kv^{-\beta +1}+2\beta (1+\beta (m-1))v^{-\beta -1}f-C_{1}v^{-\beta }(f+\left|\nabla a\right|^{2}) \\
&\quad -2\left|a\right|v^{-\beta }f-v^{-\beta +1}f-\left|\nabla b\right|^{2}v^{-\beta +1}-(2-\beta )\left( \frac{C_{1}\left|a\right|}{v}+\left|b\right| \right) v^{-\beta }f \\*
&\quad -\beta (1+(m-1)(\beta +1))v^{-\beta -1}f-2\left( \frac{C_{1}\left|a\right|}{v}+\left|b\right| \right) v^{-\beta }-2\left|a\right|v^{-\beta },
\end{align*}
    where $C_{1}=\max\left\{ e^{-1}, v_{\max}\left|\log v_{\max}\right| \right\} $. We derive from this inequality and $f\geqslant 1$ that
    \begin{align}
-\mathcal{L}(F)&\geqslant \left\{ -(m-1)\beta ^{2}+(m-2)\beta -\frac{n(m-1)}{2} \right\} v^{-\beta -1}f^{2} \label{equ;12} \\*
&\quad -\left\{ -(m-1)\beta ^{2}+(m-2)\beta -\frac{n(m-1)}{2}+(4-\beta )C_{1}\left|a\right| \right\} v^{-\beta -1}f \notag \\
&\quad -\left\{ C_{1}(1+\left|\nabla a\right|^{2})+4\left|a\right|+v+\left|\nabla b\right|^{2}v+(4-\beta )\left|b\right| \right\} v^{-\beta }f \notag \\*
&\quad -2(1+\beta (m-1))\left<\nabla F, \nabla v \right>-2(m-1)kv^{-\beta +1}f. \notag
\end{align}
    
    Now we take $\displaystyle{\beta =\beta (m, n)=-\frac{1}{2}\frac{2-m}{m-1}<0}$. Since $1<m<1+\frac{1}{1+\sqrt{2n}}$, it is easy to check that 
    \begin{equation*}
    A=-(m-1)\beta ^{2}+(m-2)\beta -\frac{n(m-1)}{2}>0. 
    \end{equation*}
    Hence we conclude from (\ref{equ;12}) that 
    \begin{equation}\label{equ;13}
    \mathcal{L}(F)+Av^{-\beta -1}f^{2}\leqslant C_{3}v^{-\beta -1}f+2(1+\beta (m-1))\left<\nabla F, \nabla v \right>+2(m-1)kv^{-\beta +1}f, 
    \end{equation}
    where $C_{3}=C_{3}(m, n, \beta , D_{1}, D_{2}, D_{3}, D_{4}, v_{\max})$ and $C_{3}\geqslant A$. 
    
    Next, let $\zeta $ be a cut-off function supported on $Q_{\rho }$. For any $l\geqslant 0$ we multiply by $\zeta ^{2}F^{l}$ on both sides of (\ref{equ;13}) and integrate on $Q_{\rho , \tau }$, then
    \begin{equation*}
    \begin{aligned}
    &\int \zeta ^{2}F^{l}F_{t}-(m-1)\int \zeta ^{2}F^{l}v\Delta F+A\int \zeta ^{2}v^{-\beta -1}F^{l}f^{2}\\
    &\quad \leqslant C_{3}\int \zeta ^{2}v^{-\beta -1}F^{l}f+2(1+\beta (m-1))\int \zeta ^{2}F^{l}\left<\nabla F, \nabla v \right>+2(m-1)k\int \zeta ^{2}v^{-\beta +1}F^{l}f. 
    \end{aligned}
    \end{equation*}
    Integrating by parts implies
    \begin{align*}
&\int \zeta ^{2}F^{l}F_{t}-(m-1)\int\zeta ^{2}F^{l}v\Delta F \\*
&\quad =\frac{1}{l+1}\int\zeta ^{2}(F^{l+1})_{t}+(m-1)\int\left<\nabla F, \nabla (\zeta ^{2}F^{l}v) \right> \\
&\quad =\frac{1}{l+1}\int_{B_{\rho }}\zeta ^{2}F^{l+1}(x, \tau )-\frac{1}{l+1}\int\frac{\partial \zeta ^{2}}{\partial t}F^{l+1} \\*
&\quad \quad +(m-1)\left\{ \int l\zeta ^{2}vF^{l-1}\left|\nabla F\right|^{2}+\int \zeta ^{2}F^{l}\left<\nabla F, \nabla v \right>+\int 2\zeta vF^{l}\left<\nabla F, \nabla \zeta  \right> . \right\}
\end{align*}
    Hence 
    \begin{equation}\label{equ;14}
    \begin{aligned}
    &\frac{1}{l+1}\int_{B_{\rho }}\zeta ^{2}F^{l+1}(x, \tau )+(m-1)l\int \zeta ^{2}vF^{l-1}\left|\nabla F\right|^{2}+A\int \zeta ^{2}v^{\beta -1}F^{l+2}\\
    &\quad \leqslant \frac{1}{l+1}\int \frac{\partial \zeta ^{2}}{\partial t}F^{l+1}+C_{3}\int \zeta ^{2}v^{-1}F^{l+1}+\int \zeta ^{2}F^{l}\left<\nabla F, \nabla v \right>\\
    &\quad \quad -2(m-1)\int \zeta vF^{l}\left<\nabla F, \nabla \zeta  \right>+2(m-1)k\int \zeta ^{2}vF^{l+1}.
    \end{aligned}
    \end{equation}
    By Cauchy-Schwarz inequality, 
    \begin{equation}\label{equ;15}
    \int \zeta ^{2}F^{l}\left<\nabla F, \nabla v \right>\leqslant \frac{(m-1)l}{4}\int \zeta ^{2}vF^{l-1}\left|\nabla F\right|^{2}+\frac{1}{(m-1)l}\int \zeta ^{2}v^{-1}F^{l+1}(f-1), 
    \end{equation}
and
    \begin{align}\label{equ;16}
    &-2(m-1)\int \zeta vF^{l}\left<\nabla F, \nabla \zeta  \right>\nonumber\\*
\leqslant &\frac{(m-1)l}{4}\int \zeta ^{2}vF^{l-1}\left|\nabla F\right|^{2}+\frac{4(m-1)}{l}\int vF^{l+1}\left|\nabla \zeta \right|^{2}. 
    \end{align}
    Combining (\ref{equ;15}) and (\ref{equ;16}) with (\ref{equ;14}) we have
    \begin{equation*}
    \begin{aligned}
    &\frac{1}{l+1}\int_{B_{\rho }}\zeta ^{2}F^{l+1}(x, \tau )+\frac{(m-1)l}{2}\int \zeta ^{2}vF^{l-1}\left|\nabla F\right|^{2}+\left\{ A-\frac{1}{(m-1)l} \right\} \int \zeta ^{2}v^{\beta -1}F^{l+2}\\
    &\quad \leqslant \frac{1}{l+1}\int \frac{\partial \zeta ^{2}}{\partial t}F^{l+1}+\left\{ C_{3}-\frac{1}{(m-1)l} \right\} \int \zeta ^{2}v^{-1}F^{l+1}\\
    &\quad \quad +\frac{4(m-1)}{l}\int vF^{l+1}\left|\nabla \zeta \right|^{2}+2(m-1)k\int \zeta ^{2}vF^{l+1}.
    \end{aligned}
    \end{equation*}
    Let $l$ satisfies $l\geqslant 2$ and $\displaystyle{A\geqslant \frac{2}{(m-1)l}}$, then $\displaystyle{C_{3}-\frac{1}{(m-1)l}>\frac{1}{2}A>0}$. We deduce further that
    \begin{equation}\label{equ;17}
    \begin{aligned}
    &\frac{1}{l+1}\int_{B_{\rho }}\zeta ^{2}F^{l+1}(x, \tau )+\frac{(m-1)l}{2}\int \zeta ^{2}vF^{l-1}\left|\nabla F\right|^{2}+\frac{1}{2}A\int \zeta ^{2}v^{\beta -1}F^{l+2}\\
    &\quad \leqslant \frac{1}{l+1}\int \frac{\partial \zeta ^{2}}{\partial t}F^{l+1}+\left\{ C_{3}-\frac{1}{(m-1)l} \right\} \int \zeta ^{2}v^{-1}F^{l+1}\\
    &\quad \quad +\frac{4(m-1)}{l}\int vF^{l+1}\left|\nabla \zeta \right|^{2}+2(m-1)k\int \zeta ^{2}vF^{l+1}.
    \end{aligned}
    \end{equation}
    Since
    \begin{equation*}
    \begin{aligned}
    v\left|\nabla (\zeta F^{(l+1)/2})\right|^{2}&=v\left|\frac{l+1}{2}\zeta F^{(l-1)/2}\nabla F+F^{(l+1)/2}\nabla \zeta \right|^{2}\\
    &\leqslant \frac{(l+1)^{2}}{2}\zeta ^{2}vF^{l-1}\left|\nabla F\right|^{2}+2vF^{l+1}\left|\nabla \zeta \right|^{2}, 
    \end{aligned}
    \end{equation*}
    we have
    \begin{align}\label{equ;18}
   & \frac{(m-1)l}{(l+1)^{2}}\int v\left|\nabla (\zeta F^{(l+1)/2})\right|^{2}\nonumber\\*
\leqslant &\frac{(m-1)l}{2}\int \zeta^{2} vF^{l-1}\left|\nabla F\right|^{2}+\frac{2(m-1)l}{(l+1)^{2}}\int vF^{l+1}\left|\nabla \zeta \right|^{2}. 
    \end{align}
    Substituting (\ref{equ;18}) into (\ref{equ;17}) yields
    \begin{equation*}
    \begin{aligned}
    &\frac{1}{l+1}\int_{B_{\rho }}\zeta ^{2}F^{l+1}(x, \tau )+\frac{(m-1)l}{(l+1)^{2}}\int v\left|\nabla (\zeta F^{(l+1)/2})\right|^{2}+\frac{1}{2}A\int \zeta ^{2}v^{\beta -1}F^{l+2}\\
    &\quad \leqslant \frac{1}{l+1}\int \frac{\partial \zeta ^{2}}{\partial t}F^{l+1}+\left\{ C_{3}-\frac{1}{(m-1)l} \right\} \int\zeta ^{2}v^{-1}F^{l+1}\\
    &\quad \quad +\left\{ \frac{4(m-1)}{l}+\frac{2(m-1)l}{(l+1)^{2}} \right\} \int vF^{l+1}\left|\nabla \zeta \right|^{2}+2(m-1)k\int \zeta ^{2}vF^{l+1}.
    \end{aligned}
    \end{equation*}
    For the function $\varphi =\zeta F^{(l+1)/2}$, we can integrate both sides of Saloff-Coste's Sobolev inequality(see (\ref{equ;sobolev})) on $(t_{0}-\rho ^{2}, \tau ]$ with respect to $t$ and deduce that
    \begin{equation*}
    \begin{aligned}
    &\quad   e^{-C_{n}(1+\sqrt{k}R)}\left|B_{R}\right|^{2/n}R^{-2}\int_{t_{0}-\rho ^{2}}^{\tau }\left\|\zeta F^{(l+1)/2}\right\|_{L^{\frac{2n}{n-2}}(B_{\rho })}^{2}  d t\\
    &\quad \quad\quad \leqslant \int_{Q_{\rho , \tau }}\left|\nabla (\zeta F^{(l+1)/2})\right|^{2}+R^{-2}\int_{Q_{\rho , \tau }}\zeta ^{2}F^{l+1}.
    \end{aligned}
    \end{equation*}
    Then 
    \begin{align*}
&\frac{1}{l+1}\int_{B_{\rho }}\zeta ^{2}F^{l+1}(x, \tau )+\frac{(m-1)lv_{\min}}{(l+1)^{2}}e^{-C_{n}(1+\sqrt{k}R)}\left|B_{R}\right|^{2/n}R^{-2}\int_{t_{0}-\rho ^{2}}^{\tau }\left\|\zeta F^{(l+1)/2}\right\|_{L^{\frac{2n}{n-2}}(B_{\rho })}^{2}  d t \\*
&\quad \quad +\frac{1}{2}A\int \zeta ^{2}v^{\beta -1}F^{l+2} \\
&\quad \leqslant \frac{1}{l+1}\int\frac{\partial \zeta ^{2}}{\partial t}F^{l+1}+\left\{ C_{3}-\frac{1}{(m-1)l} \right\} \int \zeta ^{2}v^{-1}F^{l+1}+\frac{(m-1)lv_{\min}}{(l+1)^{2}}R^{-2}\int\zeta ^{2}F^{l+1} \\*
&\quad \quad +\left\{ \frac{4(m-1)}{l}+\frac{2(m-1)l}{(l+1)^{2}} \right\} \int vF^{l+1}\left|\nabla \zeta \right|^{2}+2(m-1)k\int \zeta ^{2}vF^{l+1}.
\end{align*}
    Now we observe that each term in the above inequality is increasing with respect to $\tau $ unless $\displaystyle{\frac{1}{l+1}\int_{B_{\rho }}\zeta ^{2}F^{l+1}(x, \tau )}$. We first take $\tau =t_{0}$ on RHS of this inequality, and then take the supremum on both sides with respect to $\tau $. Obviously the supremum of LHS is no larger than twice of RHS, namely
    \begin{align*}
&\frac{(m-1)lv_{\min}}{(l+1)^{2}}e^{-C_{n}(1+\sqrt{k}R)}\left|B_{R}\right|^{2/n}R^{-2}\int_{t_{0}-\rho ^{2}}^{t_{0}}\left\|\zeta F^{(l+1)/2}\right\|_{L^{\frac{2n}{n-2}}(B_{\rho })}^{2}  d t \\*
&\quad \quad +\frac{1}{l+1}\sup_{t_{0}-\rho ^{2}\leqslant t\leqslant t_{0}}\int_{B_{\rho }}\zeta ^{2}F^{l+1}(x, t)+\frac{1}{2}A\int_{Q_{\rho }}\zeta ^{2}v^{\beta -1}F^{l+2} \\
&\quad \leqslant \frac{2}{l+1}\int_{Q_{\rho }}\frac{\partial \zeta ^{2}}{\partial t}F^{l+1}+2\left\{ C_{3}-\frac{1}{(m-1)l} \right\} \int_{Q_{\rho }} \zeta ^{2}v^{-1}F^{l+1}+\frac{2(m-1)lv_{\min}}{(l+1)^{2}}R^{-2}\int_{Q_{\rho }}\zeta ^{2}F^{l+1} \\*
&\quad \quad +\left\{ \frac{8(m-1)}{l}+\frac{4(m-1)l}{(l+1)^{2}} \right\} \int_{Q_{\rho }} vF^{l+1}\left|\nabla \zeta \right|^{2}+4(m-1)k\int_{Q_{\rho }} \zeta ^{2}vF^{l+1}.
\end{align*}
    Applying Lemma \ref{lemma;2} to $\zeta F^{(l+1)/2}$, we derive
    \begin{align*}
&\left( \frac{(m-1)lv_{\min}}{(l+1)^{2}}e^{-C_{n}(1+\sqrt{k}R)}\left|B_{R}\right|^{2/n}R^{-2}\ \right) ^{\frac{n}{n+2}}\left\|\zeta F^{(l+1)/2}\right\|_{L^{\frac{2(n+2)}{n}}(Q_{\rho })}^{2} +\frac{1}{2}A(l+1)^{\frac{2}{n+2}}\int_{Q_{\rho }}\zeta ^{2}v^{\beta -1}F^{l+2} \\*
&\quad \leqslant (l+1)^{\frac{2}{n+2}}\Bigg\{ \frac{2}{l+1}\int_{Q_{\rho }}\frac{\partial \zeta ^{2}}{\partial t}F^{l+1}+2\left\{ C_{3}-\frac{1}{(m-1)l} \right\} \int_{Q_{\rho }} \zeta ^{2}v^{-1}F^{l+1} \\
&\quad \quad +\frac{2(m-1)lv_{\min}}{(l+1)^{2}}R^{-2}\int_{Q_{\rho }}\zeta ^{2}F^{l+1}+\left\{ \frac{8(m-1)}{l}+\frac{4(m-1)l}{(l+1)^{2}} \right\} \int_{Q_{\rho }} vF^{l+1}\left|\nabla \zeta \right|^{2} \\*
&\quad \quad +4(m-1)k\int_{Q_{\rho }} \zeta ^{2}vF^{l+1}\Bigg\}.
\end{align*}
    Multiplying by $(l+1)^{\frac{n}{n+2}}$ on both sides and using $l\geqslant 2$, one concludes that
    \begin{equation}\label{equ;19}
    \begin{aligned}
    &\left( v_{\min}e^{-C_{n}(1+\sqrt{k}R)}\left|B_{R}\right|^{2/n}R^{-2} \right)^{\frac{n}{n+2}}\left\|\zeta F^{(l+1)/2}\right\| _{L^{\frac{2(n+2)}{n}}(Q_{\rho })}^{2}+Al\int_{Q_{\rho }}\zeta ^{2}v^{\beta -1}F^{l+2}\\
    &\quad \leqslant C_{4}\bigg\{\int_{Q_{\rho }}\left|\frac{\partial \zeta ^{2}}{\partial t}\right|F^{l+1}+l\int_{Q_{\rho }}\zeta ^{2}v^{-1}F^{l+1}+\int_{Q_{\rho }}vF^{l+1}\left|\nabla \zeta \right|^{2}\\
    &\quad \quad +kl\int_{Q_{\rho }}\zeta ^{2}vF^{l+1}+v_{\min}R^{-2}\int_{Q_{\rho }}\zeta ^{2}F^{l+1}  \bigg\}. 
    \end{aligned}
    \end{equation}
    Define
    \begin{equation*}
    l_{0}=\delta_{0}(1+\sqrt{k}R)\quad ,\quad \delta_{0}=\delta_{0}(m, n)=\max\left\{ C_{n}+2, \frac{2}{(m-1)A}, -\frac{1}{\beta }, n \right\} . 
    \end{equation*}
    Choose $l\geqslant l_{0}$, then
    \begin{equation}\label{equ;20}
    kR^{2}\leqslant \frac{l_{0}^{2}}{\delta_{0}^{2}}\quad \text{and} \quad \frac{v_{\min}}{l}\leqslant \frac{v_{\min}}{\delta_{0}}. 
    \end{equation}
    Combining (\ref{equ;20}) with (\ref{equ;19}) and using $v_{\min}\leqslant v\leqslant v_{\max}$ and $0<R\leqslant 1$, we have
    \begin{align}\label{equ;21}
    &\left( v_{\min}e^{-l_{0}}\left|B_{R}\right|^{2/n}R^{-2} \right) ^{\frac{n}{n+2}}\left\|\zeta F^{(l+1)/2}\right\|_{L^{\frac{2(n+2)}{n}}(Q_{\rho })}^{2}+Al\int_{Q_{\rho }}\zeta ^{2}v^{\beta -1}F^{l+2}\nonumber\\*
    \leqslant& C_{5}\bigg\{\int_{Q_{\rho }}\left( \left|\frac{\partial \zeta ^{2}}{\partial t}\right|+\left|\nabla \zeta \right|^{2} \right)F^{l+1}+lR^{-2}\int_{Q_{\rho }}\zeta ^{2}v^{-1}F^{l+1}+l_{0}^{2}lR^{-2}\int_{Q_{\rho }}\zeta ^{2}F^{l+1}  \bigg\}, 
    \end{align}
    where $C_{5}=C_{5}(m, n, \beta , D_{1}, D_{2}, D_{3}, D_{4}, v_{\max})$. 

    Next, we apply the standard iteration to (\ref{equ;21}). First ignore the second term on the left and we deduce
    \begin{equation}\label{equ;22}
    \begin{aligned}
    &\left( v_{\min}e^{-l_{0}}\left|B_{R}\right|^{2/n}R^{-2} \right) ^{\frac{n}{n+2}}\left\|\zeta F^{(l+1)/2}\right\|_{L^{\frac{2(n+2)}{n}}(Q_{\rho })}^{2}\\
    &\quad \leqslant C_{5}\left\{ \int_{Q_{\rho }}(\left|\zeta _{t}\right|+\left|\nabla \zeta \right|^{2})F^{l+1}+lR^{-2}\int_{Q_{\rho }}\zeta ^{2}v^{-1}F^{l+1}+l_{0}^{2}lR^{-2}\int_{Q_{\rho }}\zeta ^{2}F^{l+1} \right\} . 
    \end{aligned}
    \end{equation}
    For $k\geqslant 1$, choose
    \begin{equation*}
    \rho _{k}=\frac{R}{2}+\frac{R}{2^{k+1}}\quad \text{and}\quad \gamma _{k}=(l_{0}+1)\left( \frac{n+2}{n} \right) ^{k}.
    \end{equation*}
    Moreover, choose cut-off functions $\zeta _{k}=\zeta _{k}(x, t)$ such that $0\leqslant \zeta _{k}\leqslant 1$, with
    \begin{equation*}
    \zeta _{k}(x, t)=\begin{cases} 1, & (x, t)\in Q_{k+1}=Q_{\rho _{k+1}} \\
    0 , & (x, t)\in Q_{k}\backslash Q_{(\rho _{k}+\rho _{k+1})/2}, 
    \end{cases}
    \end{equation*}
    and
    \begin{equation*}
    \left|(\zeta _{k})_{t}\right|+\left|\nabla \zeta _{k}\right|^{2}\leqslant \frac{C(n)}{(\rho _{k}-\rho _{k+1})^{2}}\leqslant \frac{C(n)4^{k}}{R^{2}}. 
    \end{equation*}
    Take $\rho =\rho _{k}$ and $l+1=\gamma _{k}$ in (\ref{equ;22}), 
    \begin{equation*}
    \begin{aligned}
    &(v_{\min}e^{-l_{0}}\left|B_{R}\right|^{2/n})^{\frac{n}{n+2}}\left\|\zeta F^{\gamma _{k}/2}\right\|_{L^{\frac{2(n+2)}{n}}(Q_{k})}^{2}\\
    &\quad \leqslant \frac{C_{6}}{R^{4/(n+2)}}\left\{ 4^{k}+l_{0}^{2}\left( (l_{0}+1)\left( \frac{n+2}{n} \right)^{k}-1  \right)+\frac{1}{v_{\min}}\left( (l_{0}+1)\left( \frac{n+2}{n} \right)^{k}-1  \right)  \right\} \int_{Q_{k}}F^{\gamma _{k}}. 
    \end{aligned}
    \end{equation*}
    This implies
    \begin{equation*}
    \begin{aligned}
    &\left( \int_{Q_{k+1}} F^{\gamma _{k+1}}\right) ^{\frac{n}{n+2}}\\
    &\quad \leqslant C_{6}\left( \frac{e^{l_{0}}}{v_{\min}R^{4/n}\left|B_{R}\right|^{2/n}} \right) ^{\frac{n}{n+2}}\left\{ 4^{k}+l_{0}^{3}\left( \frac{n+2}{n} \right)^{k}+\frac{l_{0}}{v_{\min}}\left( \frac{n+2}{n} \right)^{k}   \right\} \int_{Q_{k}}F^{\gamma _{k}}\\
    &\quad \leqslant C_{7}\left( \frac{e^{l_{0}}}{v_{\min}R^{4/n}\left|B_{R}\right|^{2/n}} \right) ^{\frac{n}{n+2}}l_{0}^{3}\max\left\{ v_{\min}^{-1}, 1 \right\} 4^{k}\int_{Q_{k}}F^{\gamma _{k}}. 
    \end{aligned}
    \end{equation*}
    Taking the $1/\gamma _{k}$ power on both sides yields
    \begin{align*}
\left\|F\right\|_{L^{\gamma _{k+1}}(Q_{k+1})}&\leqslant \left[ C_{7}\left( \frac{e^{l_{0}}}{v_{\min}R^{4/n}\left|B_{R}\right|^{2/n}} \right) ^{\frac{n}{n+2}}l_{0}^{3}\max\left\{ v_{\min}^{-1}, 1 \right\}  \right] ^{\frac{1}{\gamma _{k}}}4^{\frac{k}{\gamma _{k}}}\left\|F\right\|_{L^{\gamma _{k}}(Q_{k})} \\
&\leqslant \left[ C_{7}\left( \frac{e^{l_{0}}}{v_{\min}R^{4/n}\left|B_{R}\right|^{2/n}} \right) ^{\frac{n}{n+2}}l_{0}^{3}\max\left\{ v_{\min}^{-1}, 1 \right\}  \right] ^{\sum\limits_{k=1}^{\infty}\frac{1}{\gamma _{k}}}4^{\sum\limits_{k=1}^{\infty}\frac{k}{\gamma _{k}}}\left\|F\right\|_{L^{\gamma _{1}}(Q_{3R/4})} \\
&\leqslant C_{8}\max\left\{ v_{\min}^{-\frac{n}{2(l_{0}+1)}}, 1 \right\}\left( \frac{e^{l_{0}}}{v_{\min}R^{4/n}\left|B_{R}\right|^{2/n}} \right) ^{\frac{n^{2}}{2(n+2)(l_{0}+1)}}\left\|F\right\| _{L^{\gamma_{1}}(Q_{3R/4})}.
\end{align*}
    Letting $k\to \infty$, we conclude 
    \begin{equation}\label{equ;Moser}
    \sup_{(x, t)\in Q_{R/2}}F(x, t)\leqslant C_{8}\max\left\{ v_{\min}^{-1/2}, 1 \right\} \left( \frac{e^{l_{0}}}{v_{\min}R^{4/n}\left|B_{R}\right|^{2/n}} \right) ^{\frac{n^{2}}{2(n+2)(l_{0}+1)}}\left\|F\right\| _{L^{\gamma_{1}}(Q_{3R/4})}, 
    \end{equation}
    where we have used $l\geqslant n$.

    To finish the proof, it suffices to estimate $\left\|F\right\|_{L^{\gamma_{1}}(Q_{3R/4})}$. We derive from (\ref{equ;21}) by taking $l=l_{0}$ and $\rho =R$ that
    \begin{equation}\label{equ;23}
    \begin{aligned}
    &\left( v_{\min}e^{-l_{0}}\left|B_{R}\right|^{2/n}R^{-2} \right) ^{\frac{n}{n+2}}\left\|\zeta F^{(l_{0}+1)/2}\right\|_{L^{\frac{2(n+2)}{n}}(Q_{R})}^{2}+Al_{0}\int_{Q_{R}}\zeta ^{2}v^{\beta -1}F^{l_{0}+2}\\
    &\quad \leqslant C_{5}\left\{ \int_{Q_{R}}\left( \left|\frac{\partial \zeta ^{2}}{\partial t}\right|+\left|\nabla \zeta \right|^{2} \right)F^{l_{0}+1}+l_{0}R^{-2}\int_{Q_{R}}\zeta ^{2}v^{-1}F^{l_{0}+1}+l_{0}^{3}R^{-2}\int_{Q_{R}}\zeta ^{2}F^{l_{0}+1}  \right\} . 
    \end{aligned}
    \end{equation}
    Observe that if 
    \begin{equation*}
    Fv^{\beta -1}\geqslant \frac{4C_{5}l_{0}^{2}}{AR^{2}}, 
    \end{equation*}
    then we have
    \begin{equation*}
    C_{5}l_{0}^{3}R^{-2}\int_{Q_{R}}\zeta ^{2}F^{l_{0}+1}\leqslant \frac{1}{4}Al_{0}\int_{Q_{R}}\zeta ^{2}v^{\beta -1}F^{l_{0}+2}.
    \end{equation*}
    Denote
    \begin{equation*}
    \Omega_{1}=\left\{ Fv^{\beta -1}\geqslant \frac{4C_{5}l_{0}^{2}}{AR^{2}} \right\} \quad ,\quad \Omega_{2}=\left\{ Fv^{\beta -1}<\frac{4C_{5}l_{0}^{2}}{AR^{2}} \right\} . 
    \end{equation*}
    The third term on RHS of (\ref{equ;23}) can be estimated as
    \begin{equation}\label{equ;24}
    \begin{aligned}
    C_{5}l_{0}^{3}R^{-2}\int_{Q_{R}}\zeta ^{2}F^{l_{0}+1}&=C_{5}l_{0}^{3}R^{-2}\int_{\Omega_{1}}\zeta ^{2}F^{l_{0}+1}+C_{5}l_{0}^{3}R^{-2}\int_{\Omega_{2}}\zeta ^{2}F^{l_{0}+1}\\
    &\leqslant C_{5}l_{0}^{3}R^{-2}\left( \frac{4C_{5}l_{0}^{2}v_{\max}^{1-\beta }}{AR^{2}} \right)^{l_{0}+1}R^{2}\left|B_{R}\right|  +\frac{1}{4}Al_{0}\int_{Q_{R}}\zeta ^{2}v^{\beta -1}F^{l_{0}+2}.
    \end{aligned}
    \end{equation}
    Similarly, denote
    \begin{equation*}
    \Omega_{3}=\left\{ Fv^{\beta }\geqslant \frac{4C_{5}}{AR^{2}} \right\} \quad ,\quad \Omega_{4}=\left\{ Fv^{\beta }<\frac{4C_{5}}{AR^{2}} \right\}. 
    \end{equation*}
    The second term on RHS of (\ref{equ;23}) can be estimated as
    \begin{equation}\label{equ;25}
    \begin{aligned}
        C_{5}l_{0}R^{-2}\int_{Q_{R}}\zeta ^{2}v^{-1}F^{l_{0}+1}&=C_{5}l_{0}R^{-2}\int_{\Omega_{3}}\zeta ^{2}v^{-1}F^{l_{0}+1}+C_{5}l_{0}R^{-2}\int_{\Omega_{4}}\zeta ^{2}v^{-1}F^{l_{0}+1}\\
        &\leqslant C_{5}l_{0}R^{-2}\left( \frac{4C_{5}}{AR^{2}} \right) ^{l_{0}+1}v_{\max}^{-1-\beta (l_{0}+1)}R^{2}\left|B_{R}\right| +\frac{1}{4}Al_{0}\int_{Q_{R}}\zeta ^{2}v^{\beta -1}F^{l_{0}+2},
    \end{aligned}
    \end{equation}
    where we have used $l_{0}\geqslant \delta_{0}\geqslant -\beta ^{-1}$. Substituting (\ref{equ;24}) and (\ref{equ;25}) into (\ref{equ;23}) yields
    \begin{align}
&\left( v_{\min}e^{-l_{0}}\left|B_{R}\right|^{2/n}R^{-2} \right)^{\frac{n}{n+2}}\left\|\zeta F^{(l_{0}+1)/2}\right\|_{L^{\frac{2(n+2)}{n}}(Q_{R})}^{2}+\frac{1}{2}Al_{0}\int_{Q_{R}}\zeta ^{2}v^{\beta -1}F^{l_{0}+2} \label{equ;26} \\*
&\quad \leqslant C_{5}\int_{Q_{R}}\left( \left|\frac{\partial \zeta ^{2}}{\partial t}\right|+\left|\nabla \zeta \right|^{2} \right) F^{l_{0}+1}+C_{5}l_{0}^{3}\left( \frac{4C_{5}l_{0}^{2}v_{\max}^{1-\beta }}{AR^{2}} \right) ^{l_{0}+1}\left|B_{R}\right| \notag \\*
&\quad \quad +C_{5}l_{0}\left( \frac{4C_{5}}{AR^{2}} \right) ^{l_{0}+1}v_{\max}^{-1-\beta (l_{0}+1)}\left|B_{R}\right|. \notag
\end{align}
    Let $\widetilde{\xi }=\widetilde{\xi }(t)$ and $\widetilde{\eta }=\widetilde{\eta }(x)$ be cut-off functions which satisfy $0\leqslant \xi , \eta \leqslant 1$ with
    \begin{equation*}
    \widetilde{\xi }(t)=\begin{cases} 1, & \displaystyle{t_{0}-\left( \frac{3R}{4} \right)^{2}\leqslant t\leqslant t_{0}},   \\
    0 , & \displaystyle{t\leqslant t_{0}-\left( \frac{7R}{8} \right)^{2}},
    \end{cases}
    \end{equation*}
    \begin{equation*}
    \widetilde{\eta }(x)=\begin{cases} 1, & x\in B_{3R/4}, \\
    0 , & x\notin B_{7R/8},  
    \end{cases}
    \end{equation*}
    and
    \begin{equation*}
    \left|\frac{ d \widetilde{\xi }}{ d t}\right|+\left|\nabla \widetilde{\eta }\right|^{2}\leqslant \frac{C(n)}{R^{2}}. 
    \end{equation*}
    Take $\zeta =\xi ^{1/2}\eta $, where $\xi =\widetilde{\xi }^{l_{0}+2}$, $\eta =\widetilde{\eta }^{l_{0}+2}$. We have
    \begin{equation*}
    \left|\frac{\partial \zeta ^{2}}{\partial t}\right|=(l_{0}+2)\widetilde{\xi }^{l_{0}+1}\eta ^{2}\left|\frac{ d \widetilde{\xi }}{ d t}\right|\leqslant \frac{C(n)l_{0}}{R^{2}}\xi ^{\frac{l_{0}+1}{l_{0}+2}}\eta ^{2}
    \end{equation*}
    and
    \begin{equation*}
    \left|\nabla \zeta \right|^{2}=(l_{0}+2)^{2}\widetilde{\xi }^{l_{0}+2}\widetilde{\eta }^{2l_{0}+2}\left|\nabla \widetilde{\eta }\right|^{2}\leqslant \frac{C(n)l_{0}^{2}}{R^{2}}\xi \eta ^{\frac{2(l_{0}+1)}{l_{0}+2}}. 
    \end{equation*}
    Using these inequality and Young's inequality, one can estimate the first term on RHS of (\ref{equ;26}) as follows:
    \begin{align}
C_{5}\int_{Q_{R}}\left|\frac{\partial \zeta ^{2}}{\partial t}\right|F^{l_{0}+1}&\leqslant \frac{C_{9}l_{0}}{R^{2}}\int_{Q_{R}}\xi ^{\frac{l_{0}+1}{l_{0}+2}}\eta ^{2}F^{l_{0}+1}\leqslant \frac{C_{9}l_{0}}{R^{2}}\int_{Q_{R}}(\xi \eta ^{2})^{\frac{l_{0}+1}{l_{0}+2}}F^{l_{0}+1} \label{equ;27} \\
&\leqslant \frac{C_{9}l_{0}}{R^{2}}\left( \int_{Q_{R}}\zeta ^{2}F^{l_{0}+2} \right) ^{\frac{l_{0}+1}{l_{0}+2}}\left( R^{2}\left|B_{R}\right| \right) ^{\frac{1}{l_{0}+2}} \notag \\
&\leqslant \frac{1}{l_{0}+2}\left( \frac{1}{4}Al_{0}v_{\max}^{\beta -1}\frac{l_{0}+2}{l_{0}+1} \right) ^{-(l_{0}+1)}\left( \frac{C_{9}l_{0}}{R^{2}} \right) ^{l_{0}+2}R^{2}\left|B_{R}\right|+\frac{1}{4}Al_{0}v_{\max}^{\beta -1}\int_{Q_{R}}\zeta ^{2}F^{l_{0}+2}. \notag
\end{align}
    Similarly, 
    \begin{align}
C_{5}\int_{Q_{R}}\left|\nabla \zeta \right|^{2}F^{l_{0}+1}&\leqslant \frac{C_{9}l_{0}^{2}}{R^{2}}\int_{Q_{R}}\xi \eta ^{\frac{2(l_{0}+1)}{l_{0}+2}}F^{l_{0}+1}\leqslant \frac{C_{9}l_{0}^{2}}{R^{2}}\int_{Q_{R}}(\xi \eta ^{2})^{\frac{l_{0}+1}{l_{0}+2}}F^{l_{0}+1} \label{equ;28} \\
&\leqslant \frac{C_{9}l_{0}^{2}}{R^{2}}\left( \int_{Q_{R}}\zeta ^{2}F^{l_{0}+2} \right) ^{\frac{l_{0}+1}{l_{0}+2}}\left( R^{2}\left|B_{R}\right| \right) ^{\frac{1}{l_{0}+2}} \notag \\
&\leqslant \frac{1}{l_{0}+2}\left( \frac{1}{4}Al_{0}v_{\max}^{\beta -1}\frac{l_{0}+2}{l_{0}+1} \right) ^{-(l_{0}+1)}\left( \frac{C_{9}l_{0}^{2}}{R^{2}} \right) ^{l_{0}+2}R^{2}\left|B_{R}\right|+\frac{1}{4}Al_{0}v_{\max}^{\beta -1}\int_{Q_{R}}\zeta ^{2}F^{l_{0}+2}. \notag
\end{align}
    Combining (\ref{equ;26}) with (\ref{equ;27}) and (\ref{equ;28}) implies
    \begin{align*}
&\left( v_{\min}e^{-l_{0}}\left|B_{R}\right|^{2/n}R^{-2} \right)^{\frac{n}{n+2}}\left\|\zeta F^{(l_{0}+1)/2}\right\|_{L^{\frac{2(n+2)}{n}}(Q_{R})}^{2} \\*
&\quad \leqslant C_{5}l_{0}^{3}\left( \frac{4C_{5}l_{0}^{2}v_{\max}^{1-\beta }}{AR^{2}} \right)^{l_{0}+1}\left|B_{R}\right| +C_{5}l_{0}\left( \frac{4C_{5}}{AR^{2}} \right) ^{l_{0}+1}v_{\max}^{-1-\beta (l_{0}+1)}\left|B_{R}\right| \\
&\quad \quad +\frac{2}{l_{0}+2}\left( \frac{1}{4}Al_{0}v_{\max}^{\beta -1}\frac{l_{0}+2}{l_{0}+1} \right) ^{-(l_{0}+1)}\left( \frac{C_{9}l_{0}^{2}}{R^{2}} \right) ^{l_{0}+2}R^{2}\left|B_{R}\right| \\*
&\quad \leqslant l_{0}^{3}\left( \frac{C_{10}l_{0}^{2}}{AR^{2}} \right) ^{l_{0}+1}\left|B_{R}\right|.
\end{align*}
    Since $\left. \zeta  \right|_{Q_{3R/4}}\equiv 1$, we have
    \begin{equation*}
    \left( \int_{Q_{3R/4}}F^{\gamma_{1}} \right) ^{\frac{n}{n+2}}\leqslant l_{0}^{3}\left( \frac{C_{10}l_{0}^{2}}{AR^{2}} \right) ^{l_{0}+1}\left( \frac{e^{l_{0}}R^{2}\left|B_{R}\right|}{v_{\min}} \right) ^{\frac{n}{n+2}}. 
    \end{equation*}
    Taking the $1/(l_{0}+1)$ power on both sides yields
    \begin{equation}\label{equ;29}
    \left\|F\right\|_{L^{\gamma_{1}}(Q_{3R/4})}\leqslant C_{11}\left( \frac{e^{l_{0}}R^{2}\left|B_{R}\right|}{v_{\min}} \right) ^{\frac{n}{(n+2)(l_{0}+1)}}\frac{l_{0}^{2}}{R^{2}}. 
    \end{equation}
    Finally, we combine (\ref{equ;29}) with (\ref{equ;Moser}) and derive
    \begin{equation*}
    \sup_{(x, t)\in Q_{R/2}}F(x, t)\leqslant C_{12}\max\left\{ v_{\min}^{-1}, 1 \right\} \left( \frac{1+\sqrt{k}R}{R} \right) ^{2}. 
    \end{equation*}
    Since $\displaystyle{F=v^{-\beta }(\left|\nabla v\right|^{2}+1)=v^{\frac{2-m}{2(m-1)}}\left|\nabla v\right|^{2}+v^{\frac{2-m}{2(m-1)}}}$, this completes the proof. 
\end{proof}
\begin{remark}
    Throughout the proof, the condition $0<R\leqslant 1$ is only used to give (\ref{equ;21}).
\end{remark}

A technical linear algebra lemma for symmetric matrix is required. 
\begin{lemma}[\cite{X}]\label{lemma;4}
Let $A=(a_{ij})$ be a nonzero $n\times n$ symmetric matrix. Then for any $a, b\in \mathbb{R}$, there holds
\begin{equation}\label{equ;37}
\max_{A\in \mathcal{S}(n);\left|e\right|=1}\left( a\frac{A(e, e)}{\left|A\right|}+b\frac{\mathrm{tr}\,A}{\left|A\right|} \right) ^{2}=(a+b)^{2}+(n-1)b^{2}, 
\end{equation}
where $\mathcal{S}(n)$ is the set of all the n-dimensional symmetric matrix. 
\end{lemma}
\begin{theorem}\label{thm;2}
Let $(M, g)$ be a Riemannian manifold of dimension $n\geqslant 2$ without boundary and with $\mathrm{Ric}\,_{M}\geqslant -k$ for some $k\geqslant 0$. Suppose that $u$ is any positive solution to (\ref{equ;1}) on $Q_{R}=B_{R}\times (t_{0}-R^{2}, t_{0}]$,  where $1<m<1+\frac{1}{\sqrt{n-1}}$ and $a$, $b$ are constants. Suppose also that $0<u\leqslant M$. Let $v=m u^{m-1}/(m-1)\leqslant V$. If one of the following conditions holds, 
\begin{equation}\label{equ;30}
\begin{aligned}
a>0,& \quad  a\log u_{\max}+\frac{2}{2m-1}a+b< 0, \\
a=0,& \quad b\leqslant 0, \\
a<0,& \quad  a\log u_{\min}+\frac{2}{2m-1}a+b< 0. 
\end{aligned}
\end{equation}
Then we have
\begin{equation}\label{equ;31}
\sup_{(x, t)\in Q_{R/2}}v^{\frac{1}{2(m-1)}}\left|\nabla v\right|  \leqslant C\frac{1+\sqrt{k}R}{R}, 
\end{equation}
where $C=C(m, n, a,b , V )$ is a constant . In particular, $u$ does not need to be bounded from below when $a\geqslant 0$. 
\end{theorem}

\begin{proof}
    We use the notation in the proof of Theorem \ref{thm;R1} and only prove for $n\geqslant 3$. The spatial Sobolev exponent, the parabolic exponent, the Moser amplification factor and the exponent $\lambda$ defined below are $2n'/(n'-2)$, $2(n'+2)/n'$, $(n'+2)/n'$ and $(n'+2)/2$, respectively. The argument is also valid when $n=2$ as long as all the $n$ in these exponents  are replaced by $n'$.

    Let 
    \begin{equation*}
    f=\left|\nabla v\right|^{2}\quad \text{and} \quad F=v^{-\beta }f=\frac{\left|\nabla v\right|^{2}}{v^{\beta }}, 
    \end{equation*}
    where $\beta $ is a (negative) constant to be determined. Then we have
    \begin{equation*}
    \begin{aligned}
    \mathcal{L}(F)&=-\beta (1+(m-1)(\beta +1))v^{-\beta -1}f^{2}-\beta (a\log v+b')v^{-\beta }f-2(m-1)v^{-\beta +1}\left|\nabla ^{2}v\right|^{2}\\
    &\quad -2(m-1)v^{-\beta +1}\mathrm{Ric}\,(\nabla v, \nabla v)+2(m-1)v^{-\beta }f\Delta v+2v^{-\beta }\left<\nabla v, \nabla f \right>\\
    &\quad +2(a\log v+b')v^{-\beta }f+2av^{-\beta }f+2(m-1)\beta v^{-\beta }\left<\nabla v, \nabla f \right>.
    \end{aligned}
    \end{equation*}
    For some $\varepsilon $ to be determined later, we have
    \begin{align}\label{equ;38}
    & -\mathcal{L}(F)+\varepsilon \left<\nabla F, \nabla v \right>\nonumber\\*
    &=\beta (1+(m-1)(\beta +1)-\varepsilon )v^{-\beta -1}f^{2}-((2-\beta )(a\log v+b')+2a)v^{-\beta }f\nonumber\\
    &\quad +2(m-1)v^{-\beta +1}v_{ij}^{2}-2(m-1)v^{-\beta }fv_{jj}+(2\varepsilon -4(1+\beta (m-1)))v^{-\beta }v_{i}v_{j}v_{ij}\\*
    &\quad +2(m-1)v^{-\beta +1}\mathrm{Ric}\,(\nabla v, \nabla v).\nonumber
    \end{align}
    Define $A=(v_{ij})$, $e=\displaystyle{\frac{\nabla v}{\left|\nabla v\right|}}$. We use the inequality $ax^{2}+bx\geqslant \displaystyle{-\frac{b^{2}}{4a}}$ to deduce
    \begin{align}\label{equ;39}
    &2(m-1)v^{-\beta +1}v_{ij}^{2}-2(m-1)v^{-\beta }fv_{jj}+(2\varepsilon -4(1+\beta (m-1)))v^{-\beta }v_{i}v_{j}v_{ij}\nonumber\\*
    =&2(m-1)v^{-\beta +1}\left|A\right|^{2}-2(m-1)v^{-\beta }f\frac{\mathrm{tr}\,A}{\left|A\right|}\left|A\right|\nonumber\\
   &+(2\varepsilon -4(1+\beta (m-1)))v^{-\beta }f\frac{A(e, e)}{\left|A\right|}\left|A\right|\\*
    \geqslant & -\frac{1}{2(m-1)}\left[ (\varepsilon -2(1+\beta (m-1)))\frac{A(e, e)}{\left|A\right|}-(m-1)\frac{\mathrm{tr}\,A}{\left|A\right|} \right] ^{2}v^{-\beta -1}f^{2}. \nonumber
    \end{align}
    Combining (\ref{equ;38}) with (\ref{equ;39}), we derive
    \begin{equation*}
    \begin{aligned}
    &-\mathcal{L}(F)+\varepsilon \left<\nabla F, \nabla v \right>+((2-\beta )(a\log v+b')+2a)v^{-\beta }f\\
    &\quad \geqslant -\frac{1}{2(m-1)}\left[ (\varepsilon -2(1+\beta (m-1)))\frac{A(e, e)}{\left|A\right|}-(m-1)\frac{\mathrm{tr}\,A}{\left|A\right|} \right] ^{2}v^{-\beta -1}f^{2}\\
    &\quad +\beta (1+(m-1)(\beta +1)-\varepsilon )v^{-\beta -1}f^{2}-2(m-1)kv^{-\beta +1}f. 
    \end{aligned}
    \end{equation*}
    Applying Lemma \ref{lemma;4}, we have
    \begin{align}
&-\mathcal{L}(F)+\varepsilon \left<\nabla F, \nabla v \right>+((2-\beta )(a\log v+b')+2a)v^{-\beta }f \label{equ;41} \\*
&\quad \geqslant -\frac{1}{2(m-1)}\bigg\{ (\varepsilon -2(1+\beta (m-1))-(m-1))^{2}+(n-1)(m-1)^{2} \notag \\
&\quad \quad -2(m-1)\beta (1+(m-1)(\beta +1)-\varepsilon ) \bigg\} v^{-\beta -1}f^{2}-2(m-1)kv^{-\beta +1}f \notag \\
&\quad =-\frac{1}{2(m-1)}\bigg\{  2(m-1)^{2}\beta ^{2}+2(m-1)(m+2-\varepsilon )\beta \notag \\*
&\quad \quad+(m+1-\varepsilon )^{2}+(n-1)(m-1)^{2}  \bigg\} v^{-\beta -1}f^{2}-2(m-1)kv^{-\beta +1}f. \notag
\end{align}
    We minimize the coefficient of $v^{-\beta -1}f^{2}$ by letting
\begin{equation*}
\beta =-\frac{1}{m-1}\quad \text{and}\quad \varepsilon =m.
\end{equation*}
Then \eqref{equ;41} becomes
\begin{equation}\label{equ;68}
\mathcal{L}(F)+Av^{-\beta -1}f^{2}\leqslant m\left<\nabla F, \nabla v \right>+2(m-1)kvF+\big((2-\beta )(a\log v+b')+2a\big)F,
\end{equation}
where
\begin{equation*}
A=\frac{1-(n-1)(m-1)^{2}}{2(m-1)}>0.
\end{equation*}

The rest of the proof is similar to that of Theorem \ref{thm;R1}. Let $\lambda >1$ be a constant. Multiplying both sides of \eqref{equ;68} by $\zeta ^{2}v^{-\lambda }F^{l}$ and integrating on $Q_{\rho ,\tau }$, we deduce
\begin{align*}
&\frac{1}{l+1}\int_{B_{\rho }}\zeta ^{2}v^{-\lambda }F^{l+1}(x,\tau )+(m-1)l\int \zeta ^{2}v^{1-\lambda }F^{l-1}\left|\nabla F\right|^{2}+A\int \zeta ^{2}v^{\beta -\lambda -1}F^{l+2} \\*
&\quad \leqslant \frac{1}{l+1}\int \frac{\partial \zeta ^{2}}{\partial t}v^{-\lambda }F^{l+1}+\big((\lambda -1)(m-1)+m\big)\int \zeta ^{2}v^{-\lambda }F^{l}\left<\nabla F,\nabla v\right> \\
&\quad\quad -2(m-1)\int \zeta v^{1-\lambda }F^{l}\left<\nabla F,\nabla \zeta \right>-\frac{\lambda }{l+1}\int \zeta ^{2}v^{-\lambda -1}F^{l+1}v_{t} \\*
&\quad\quad +2(m-1)k\int \zeta ^{2}v^{1-\lambda }F^{l+1}+\big((2-\beta )(a\log v+b')+2a\big)\int \zeta ^{2}v^{-\lambda }F^{l+1}.
\end{align*}
Since $v_{t}=(m-1)v\Delta v+\left|\nabla v\right|^{2}+av\log v+b'v$, integrating by parts gives
\begin{align*}
-\frac{\lambda }{l+1}\int \zeta ^{2}v^{-\lambda -1}F^{l+1}v_{t}
&=\frac{2\lambda (m-1)}{l+1}\int \zeta v^{-\lambda }F^{l+1}\left<\nabla \zeta ,\nabla v\right> \\*
&\quad +\lambda (m-1)\int \zeta ^{2}v^{-\lambda }F^{l}\left<\nabla F,\nabla v\right> \\
&\quad -\frac{\lambda (1+\lambda (m-1))}{l+1}\int \zeta ^{2}v^{\beta -\lambda -1}F^{l+2} \\*
&\quad -\frac{\lambda }{l+1}\int \zeta ^{2}v^{-\lambda }F^{l+1}(a\log v+b').
\end{align*}
Hence
\begin{align}
&\frac{1}{l+1}\int_{B_{\rho }}\zeta ^{2}v^{-\lambda }F^{l+1}(x,\tau )+(m-1)l\int \zeta ^{2}v^{1-\lambda }F^{l-1}\left|\nabla F\right|^{2} \label{equ;69} \\*
&\quad +\left(A+\frac{\lambda (1+\lambda (m-1))}{l+1}\right)\int \zeta ^{2}v^{\beta -\lambda -1}F^{l+2} \notag \\
&\leqslant \frac{1}{l+1}\int \frac{\partial \zeta ^{2}}{\partial t}v^{-\lambda }F^{l+1}+(1+2\lambda (m-1))\int \zeta ^{2}v^{-\lambda }F^{l}\left<\nabla F,\nabla v\right> \notag \\
&\quad -2(m-1)\int \zeta v^{1-\lambda }F^{l}\left<\nabla F,\nabla \zeta \right>+\frac{2\lambda (m-1)}{l+1}\int \zeta v^{-\lambda }F^{l+1}\left<\nabla \zeta ,\nabla v\right> \notag \\*
&\quad +2(m-1)k\int \zeta ^{2}v^{1-\lambda }F^{l+1} +\int \zeta ^{2}\left[\left(2-\beta -\frac{\lambda }{l+1}\right)(a\log v+b')+2a\right]v^{-\lambda }F^{l+1}. \notag
\end{align}

By the Cauchy--Schwarz inequality,
\begin{equation}\label{equ;70}
\begin{aligned}
&(1+2\lambda (m-1))\int \zeta ^{2}v^{-\lambda }F^{l}\left<\nabla F,\nabla v\right>\\
&\quad \leqslant \frac{(m-1)l}{4}\int \zeta ^{2}v^{1-\lambda }F^{l-1}\left|\nabla F\right|^{2}+\frac{(1+2\lambda (m-1))^{2}}{(m-1)l}\int \zeta ^{2}v^{\beta -\lambda -1}F^{l+2},
\end{aligned}
\end{equation}
\begin{equation}\label{equ;71}
\begin{aligned}
&-2(m-1)\int \zeta v^{1-\lambda }F^{l}\left<\nabla F,\nabla \zeta \right>\\
&\quad \leqslant \frac{(m-1)l}{4}\int \zeta ^{2}v^{1-\lambda }F^{l-1}\left|\nabla F\right|^{2}+\frac{4(m-1)}{l}\int v^{1-\lambda }F^{l+1}\left|\nabla \zeta \right|^{2},
\end{aligned}
\end{equation}
and
\begin{equation}\label{equ;72}
\begin{aligned}
&\frac{2\lambda (m-1)}{l+1}\int \zeta v^{-\lambda }F^{l+1}\left<\nabla \zeta ,\nabla v\right>\\
&\quad \leqslant \frac{2\lambda (m-1)}{l+1}\int \zeta v^{\beta /2-\lambda }F^{l+3/2}\left|\nabla \zeta \right|\\
&\quad \leqslant \frac{\lambda (m-1)}{l+1}\int \zeta ^{2}v^{\beta -\lambda -1}F^{l+2}+\frac{\lambda (m-1)}{l+1}\int v^{1-\lambda }F^{l+1}\left|\nabla \zeta \right|^{2}.
\end{aligned}
\end{equation}
Substituting \eqref{equ;70}--\eqref{equ;72} into \eqref{equ;69} yields
\begin{align*}
&\frac{1}{l+1}\int_{B_{\rho }}\zeta ^{2}v^{-\lambda }F^{l+1}(x,\tau )+\frac{(m-1)l}{2}\int \zeta ^{2}v^{1-\lambda }F^{l-1}\left|\nabla F\right|^{2} \\*
&\quad +\left(A+\frac{\lambda (1+(\lambda -1)(m-1))}{l+1}-\frac{(1+2\lambda (m-1))^{2}}{(m-1)l}\right)\int \zeta ^{2}v^{\beta -\lambda -1}F^{l+2} \\
&\leqslant \frac{1}{l+1}\int \frac{\partial \zeta ^{2}}{\partial t}v^{-\lambda }F^{l+1}+\left(\frac{4(m-1)}{l}+\frac{\lambda (m-1)}{l+1}\right)\int v^{1-\lambda }F^{l+1}\left|\nabla \zeta \right|^{2} \\*
&\quad +2(m-1)k\int \zeta ^{2}v^{1-\lambda }F^{l+1}+\int \zeta ^{2}\left[\left(2-\beta -\frac{\lambda }{l+1}\right)(a\log v+b')+2a\right]v^{-\lambda }F^{l+1}.
\end{align*}
Note that \eqref{equ;30} implies
\begin{equation*}
(2-\beta )(a\log v+b')+2a\leqslant0,
\end{equation*}
and the inequality is strict unless $a=b=0$. Since $\lambda >1$ is a constant, there exists $\delta _{1}>0$ such that
\begin{equation*}
\left(2-\beta -\frac{\lambda }{l+1}\right)(a\log v+b')+2a\leqslant 0
\end{equation*}
for all $l\geqslant \delta _{1}$ (this is trivial when $a=b=0$). Let
\begin{equation*}
l_{0}=\delta _{0}(1+\sqrt{k}R),\qquad \delta _{0}=\delta _{0}(m,n,\delta_{1}).
\end{equation*}
Then $l_{0}\geqslant \delta _{0}$. For $l\geqslant \frac{2(1+2\lambda (m-1))^{2}}{(m-1)A}$,
we have
\begin{equation*}
\begin{aligned}
&\frac{1}{l+1}\int_{B_{\rho }}\zeta ^{2}v^{-\lambda }F^{l+1}(x,\tau )+\frac{(m-1)l}{2}\int \zeta ^{2}v^{1-\lambda }F^{l-1}\left|\nabla F\right|^{2}+\frac{A}{2}\int \zeta ^{2}v^{\beta -\lambda -1}F^{l+2}\\
&\leqslant \frac{1}{l+1}\int \frac{\partial \zeta ^{2}}{\partial t}v^{-\lambda }F^{l+1}+\left(\frac{4(m-1)}{l}+\frac{\lambda (m-1)}{l+1}\right)\int v^{1-\lambda }F^{l+1}\left|\nabla \zeta \right|^{2}\\
&\quad +2(m-1)k\int \zeta ^{2}v^{1-\lambda }F^{l+1}.
\end{aligned}
\end{equation*}
We use the following inequality:
\begin{equation*}
\begin{aligned}
\left|\nabla \left(\zeta v^{(1-\lambda )/2}F^{(l+1)/2}\right)\right|^{2}
&\leqslant 3v^{1-\lambda }F^{l+1}\left|\nabla \zeta \right|^{2}+\frac{3(\lambda -1)^{2}}{4}\zeta ^{2}v^{\beta -\lambda -1}F^{l+2}\\
&\quad +\frac{3(l+1)^{2}}{4}\zeta ^{2}v^{1-\lambda }F^{l-1}\left|\nabla F\right|^{2},
\end{aligned}
\end{equation*}
and
\begin{equation*}
\frac{1}{l+1}\int_{B_{\rho }}\zeta ^{2}v^{1-\lambda }F^{l+1}(x,\tau )\leqslant \frac{V}{l+1}\int_{B_{\rho }}\zeta ^{2}v^{-\lambda }F^{l+1}(x,\tau )
\end{equation*}
since $0<v\leqslant V=mM^{m-1}/(m-1)$. We obtain
\begin{align*}
&\frac{1}{(l+1)V}\int_{B_{\rho }}\zeta ^{2}v^{1-\lambda }F^{l+1}(x,\tau )+\frac{2(m-1)l}{3(l+1)^{2}}\int \left|\nabla \left(\zeta v^{(1-\lambda )/2}F^{(l+1)/2}\right)\right|^{2} \\*
&\quad +\left(\frac{A}{2}-\frac{(m-1)l(\lambda -1)^{2}}{2(l+1)^{2}}\right)\int \zeta ^{2}v^{\beta -\lambda -1}F^{l+2} \\
&\leqslant \frac{1}{l+1}\int \frac{\partial \zeta ^{2}}{\partial t}v^{-\lambda }F^{l+1}+\left(\frac{4(m-1)}{l}+\frac{\lambda (m-1)}{l+1}+\frac{2(m-1)l}{(l+1)^{2}}\right)\int v^{1-\lambda }F^{l+1}\left|\nabla \zeta \right|^{2} \\*
&\quad +2(m-1)k\int \zeta ^{2}v^{1-\lambda }F^{l+1}.
\end{align*}
For $l\geqslant \frac{2(m-1)(\lambda -1)^{2}}{A}$, as in the proof of  Theorem \ref{thm;R1}, Saloff--Coste's Sobolev inequality and Lemma \ref{lemma;2} imply
\begin{equation}\label{equ;73}
\begin{aligned}
&\left(Ve^{-l_{0}}\left|B_{R}\right|^{2/n}R^{-2}\right)^{\frac{n}{n+2}}\left\|\zeta v^{(1-\lambda )/2}F^{(l+1)/2}\right\|_{L^{\frac{2(n+2)}{n}}(Q_{\rho })}^{2}+AlV\int_{Q_{\rho }}\zeta ^{2}v^{\beta -\lambda -1}F^{l+2}\\
&\leqslant C_{1}V\bigg\{\int_{Q_{\rho }}\frac{\partial \zeta ^{2}}{\partial t}v^{-\lambda }F^{l+1}+(\lambda +1)\int_{Q_{\rho }}v^{1-\lambda }F^{l+1}\left|\nabla \zeta \right|^{2} +l_{0}^{2}lR^{-2}\int_{Q_{\rho }}\zeta ^{2}v^{1-\lambda }F^{l+1}\bigg\}.
\end{aligned}
\end{equation}
Take $\lambda $ such that
\begin{equation*}
(1-\lambda )\frac{n+2}{n}=-\lambda ,
\end{equation*}
that is , $\lambda =(n+2)/2$. When $n=2$, we take $\lambda =(n'+2)/2=3$. In addition, set
\begin{equation}\label{equ;74}
\delta _{0}=\max\left\{C_{n'}+2,\ n',\ \delta_{1},\ \frac{2(1+2\lambda (m-1))^{2}}{(m-1)A},\ \frac{2(m-1)(\lambda -1)^{2}}{A},\ \frac{\lambda}{1-\beta}\right\},
\end{equation}
where $C_{n'}$ is the constant in Lemma \ref{lemma;3}.

On the one hand, omitting the highest-order term in \eqref{equ;73} and performing the standard iteration with respect to the measure $ d \nu =v^{-\lambda } d \mu d t$, we obtain
\begin{equation*}
\begin{aligned}
\left(\int_{Q_{k+1}}v^{-\lambda }F^{\gamma _{k+1}}\right)^{1/\gamma _{k+1}}
&\leqslant C_{2}\left[\left(\frac{e^{l_{0}}}{R^{4/n}\left|B_{R}\right|^{2/n}}\right)^{\frac{n}{n+2}}l_{0}^{3}4^{k}\right]^{\sum\limits_{k}\frac{1}{\gamma _{k}}}\left(\int_{Q_{3R/4}}v^{-\lambda }F^{\gamma _{1}}\right)^{1/\gamma _{1}},
\end{aligned}
\end{equation*}
where $\gamma _{k}$ and $Q_{k}$ are defined as in the proof of Theorem \ref{thm;R1}. Thus
\begin{equation}\label{equ;75}
\sup_{Q_{R/2}}F\leqslant C_{3}\left(\frac{e^{l_{0}}}{R^{4/n}\left|B_{R}\right|^{2/n}}\right)^{\frac{n^{2}}{2(n+2)(l_{0}+1)}}\left(\int_{Q_{3R/4}}v^{-\lambda }F^{\gamma _{1}}\right)^{1/\gamma _{1}},
\end{equation}
where $C_{3}=C_{3}(m,n,a,b,V)$.

We next control the initial integral. In \eqref{equ;73}, take $l=l_{0}$ and $\rho =R$, and let $\zeta =\eta ^{l_{0}+2}$, where
\begin{equation*}
0\leqslant \eta \leqslant 1,\qquad \eta \equiv 1\ \text{in }Q_{3R/4},\qquad \mathrm{Supp}\,\eta \subset Q_{R},\qquad \left|\nabla \eta \right|\leqslant \frac{C}{R},\qquad \left|\eta _{t}\right|\leqslant \frac{C}{R^{2}}.
\end{equation*}
Then
\begin{equation*}
\left|\nabla \zeta \right|^{2}=(l_{0}+2)^{2}\eta ^{2(l_{0}+1)}\left|\nabla \eta \right|^{2},\qquad \left|\frac{\partial \zeta ^{2}}{\partial t}\right|=2(l_{0}+2)\eta ^{2l_{0}+3}\left|\eta _{t}\right|.
\end{equation*}
Denote 
$H=\zeta ^{2}v^{\beta -\lambda -1}F^{l_{0}+2}=\eta ^{2(l_{0}+2)}v^{\beta -\lambda -1}F^{l_{0}+2}$. By Young's inequality,
\begin{align}
C_{1}V\int_{Q_{R}}\frac{\partial \zeta ^{2}}{\partial t}v^{-\lambda }F^{l_{0}+1}
&\leqslant \frac{2(l_{0}+2)C_{1}V}{R^{2}}\int_{Q_{R}}H^{\frac{l_{0}+1}{l_{0}+2}}v^{\frac{(l_{0}+1)(1-\beta )-\lambda }{l_{0}+2}}\eta \label{equ;76} \\*
&\leqslant \frac{1}{6}Al_{0}V\int_{Q_{R}}H \notag \\
&\quad +\frac{1}{l_{0}+2}\left(\frac{1}{6}Al_{0}V\frac{l_{0}+2}{l_{0}+1}\right)^{-(l_{0}+1)}\left(\frac{2(l_{0}+2)C_{1}V}{R^{2}}\right)^{l_{0}+2} \notag \\*
&\qquad \times \int_{Q_{R}}\eta ^{l_{0}+2}v^{(l_{0}+1)(1-\beta )-\lambda }. \notag
\end{align}
Similarly,
\begin{align}
C_{1}V(\lambda +1)\int_{Q_{R}}v^{1-\lambda }F^{l_{0}+1}\left|\nabla \zeta \right|^{2}
&\leqslant \frac{(l_{0}+2)^{2}C_{1}V(\lambda +1)}{R^{2}}\int_{Q_{R}}H^{\frac{l_{0}+1}{l_{0}+2}}v^{\frac{(l_{0}+1)(2-\beta )+1-\lambda }{l_{0}+2}} \label{equ;77} \\*
&\leqslant \frac{1}{6}Al_{0}V\int_{Q_{R}}H \notag \\
&\quad +\frac{1}{l_{0}+2}\left(\frac{1}{6}Al_{0}V\frac{l_{0}+2}{l_{0}+1}\right)^{-(l_{0}+1)} \notag \\*
&\qquad \times \left(\frac{(l_{0}+2)^{2}C_{1}V(\lambda +1)}{R^{2}}\right)^{l_{0}+2}\int_{Q_{R}}v^{(l_{0}+1)(2-\beta )+1-\lambda }. \notag
\end{align}
Note that $(l_{0}+1)(1-\beta )-\lambda \geqslant 1$ since $l_{0}\geqslant \delta _{0}\geqslant \lambda /(1-\beta )$. We control the last term in \eqref{equ;73} by
\begin{equation}\label{equ;78}
\begin{aligned}
C_{1}Vl_{0}^{3}R^{-2}\int_{Q_{R}}\zeta ^{2}v^{1-\lambda }F^{l_{0}+1}
&\leqslant \frac{1}{6}Al_{0}V\int_{Q_{R}}\zeta ^{2}v^{\beta -\lambda -1}F^{l_{0}+2}\\
&\quad +C_{1}Vl_{0}^{3}R^{-2}\int_{Q_{R}}v^{1-\lambda }\left(\frac{6C_{1}l_{0}^{2}v^{2-\beta }}{AR^{2}}\right)^{l_{0}+1},
\end{aligned}
\end{equation}
as in \eqref{equ;24}. Combining \eqref{equ;73}, for $l=l_{0}$ and $\rho =R$, with \eqref{equ;76}--\eqref{equ;78}, we have
\begin{equation*}
\begin{aligned}
&\left(Ve^{-l_{0}}\left|B_{R}\right|^{2/n}R^{-2}\right)^{\frac{n}{n+2}}\left\|\zeta v^{(1-\lambda )/2}F^{(l_{0}+1)/2}\right\|_{L^{\frac{2(n+2)}{n}}(Q_{R})}^{2}\\
&\quad \leqslant l_{0}^{3}\left(\frac{C_{4}l_{0}^{2}}{AR^{2}}\right)^{l_{0}+1}\left|B_{R}\right|,
\end{aligned}
\end{equation*}
where $C_{4}=C_{4}(m,n,a,b,V)$. Since $\zeta \equiv 1$ in $Q_{3R/4}$, this implies
\begin{equation}\label{equ;79}
\left(\int_{Q_{3R/4}}v^{-\lambda }F^{\gamma _{1}}\right)^{1/\gamma _{1}}\leqslant \left(\frac{e^{l_{0}}R^{2}\left|B_{R}\right|}{V}\right)^{\frac{n}{(n+2)(l_{0}+1)}}\frac{C_{5}l_{0}^{2}}{AR^{2}}.
\end{equation}
Combining \eqref{equ;75} and \eqref{equ;79}, we conclude that
\begin{equation*}
\sup_{(x,t)\in Q_{R/2}}F(x,t)\leqslant C_{6}\frac{l_{0}^{2}}{R^{2}}\leqslant C_{7}\left(\frac{1+\sqrt{k}R}{R}\right)^{2}.
\end{equation*}
Since $F=v^{-\beta }\left|\nabla v\right|^{2}=v^{1/(m-1)}\left|\nabla v\right|^{2}$, this completes the proof.
\end{proof}
\begin{remark}
We do not require $0<R\leqslant 1$ when \eqref{equ;30} holds. Hence a Liouville-type theorem can be proved by the above gradient estimate.
\end{remark}

\begin{remark}
One of our main purposes is to obtain a Liouville-type theorem, so we take $Q_{R}=B_{R}\times (t_{0}-R^{2},t_{0}]$. If a sharper estimate is required, one may consider $Q_{R,T}=B_{R}\times (t_{0}-T,t_{0}]$. When $a=b=0$, one can prove that the gradient estimate has the form
\begin{equation}\label{equ;80}
v^{\frac{1}{2(m-1)}}\left|\nabla v\right|\leqslant C(m,n)V^{\frac{m}{2(m-1)}}\left(\frac{1}{\sqrt{T}}+V^{1/2}\left(\frac{1}{R}+\sqrt{k}\right)\right),
\end{equation}
where $0<v\leqslant V$. This reduces to Huang-Xu-Zeng's result \cite{HXZ}. 
\end{remark}
\begin{remark}
The condition (\ref{equ;30}) in Theorem \ref{thm;2} cannot in
general be removed. Consider the equation on $\mathbb{R}^{2}$
\begin{equation*}
    u_{t}=\Delta u^{3/2}+\frac{3}{2}u\log u.
\end{equation*}
Let $w=w(s)$ be the solution of
    $w''+w^{2/3}\log w=0$,
    $w(0)=1$,
    $w'(0)=\frac{1}{2}$.
Define
\begin{equation*}
    g(s)=\int_{1}^{s}t^{2/3}\log t\,dt
    =\frac{3}{5}s^{5/3}\log s
    -\frac{9}{25}s^{5/3}
    +\frac{9}{25}.
\end{equation*}
Since $g'(s)=s^{2/3}\log s$, the function $g$ attains its
minimum at $s=1$, and $g(1)=0$. On the other hand, the equation
for $w$ can be written as $w''+g'(w)=0$. Consequently,
\begin{equation}\label{equ;81}
    \frac{1}{2}(w')^{2}+g(w)
    \equiv
    \frac{1}{2}\bigl(w'(0)\bigr)^{2}+g\bigl(w(0)\bigr)
    =\frac{1}{8}.
\end{equation}
In particular,
    $g(w)\leqslant 1/8$.
Moreover,
\begin{equation*}
    \lim_{s\to 0}g(s)=\frac{9}{25}>\frac{1}{8},
    \qquad
    \lim_{s\to+\infty}g(s)=+\infty.
\end{equation*}
Thus there exist $s_{1}\in(0,1)$ and $s_{2}>1$ such that $g(s_{1})=g(s_{2})=1/8$, 
and \eqref{equ;81} implies $0<s_{1}\leqslant w(s)\leqslant s_{2}$ for all $s\in \mathbb{R}^{n}$ from the theory of ODE. Let  $u(x,y,t)=w(x)^{2/3}$. $u$ is a positive bounded solution of the above porous medium equation. In this example, condition (\ref{equ;30}) is equivalent to $u_{\max}<e^{-1}$, while $u(0,y,t)=w(0)^{2/3}=1$. It is easy to check that $u$ does not satisfy the gradient estimate in Theorem \ref{thm;2}.
\end{remark}

\begin{proof}[Proof of Corollary \ref{cor;1}]
    The most important fact is: under the assumptions of Theorem \ref{thm;2}, $u(\cdot,t)$ is spatially constant for every $t$. Thus, solving the resulting ordinary differential equation for $u$, we obtain
\begin{equation*}
\log u(x,t)=\begin{cases}C+b(t-T_0),&a=0,\\\displaystyle{-\frac{b}{a}+C e^{a(t-T_{0})}},&a\neq 0,
\end{cases}
\end{equation*}
where $C$ is a constant. We then distinguish the different cases according to the values of $a$, $b$, and the desired conclusion follows.
\end{proof}

\vspace{.2in}


\section{FDE Type Equations}\label{FDE}
\vspace{.1in}

In this section, we consider (\ref{equ;1}) for $0<m<1$. Let $v=mu^{m-1}/(1-m)$, then direct computation yields 
\begin{equation}\label{equ;43}
v_{t}-(1-m)v\Delta v=-\left|\nabla v\right|^{2}+av\log v+b'v, 
\end{equation} 
where
\begin{equation*}
b'=(m-1)b-a\log \frac{m}{1-m}. 
\end{equation*}
We introduce the differential operator  from \cite{L}
\begin{equation*}
\mathcal{L}:=\frac{\partial }{\partial t}-(1-m)v\Delta . 
\end{equation*}
We give some formulae of $\mathcal{L}$ without proof. 
\begin{lemma}\label{lemma;5}
Let $u$ be a positive solution to (\ref{equ;1}) on Riemannian manifold $(M, g)$ of dimension $n$ for some $0<m<1$. Let $v=mu ^{m-1}/(1-m)$ and $\beta $ be any constant ($\beta \neq 1$). Then we have
\begin{equation}\label{equ;44}
\mathcal{L}(v)=-\left|\nabla v\right|^{2}+av\log v+b'v, 
\end{equation}
\begin{equation}\label{equ;45}
\mathcal{L}(v^{-\beta })=\beta (1-(1-m)(\beta +1))v^{-\beta -1}\left|\nabla v\right|^{2}-\beta v^{-\beta }(a\log v+b'), 
\end{equation}
\begin{equation}\label{equ;46}
\begin{aligned}
\mathcal{L}(\left|\nabla v\right|^{2})&=-2(1-m)v\left|\nabla ^{2}v\right|^{2}-2(1-m)v\mathrm{Ric}\,(\nabla v, \nabla v)+2(1-m)\Delta v \left|\nabla v\right|^{2}\\
&\quad -2\left<\nabla v, \nabla \left|\nabla v\right|^{2} \right>+2(a\log v+b')\left|\nabla v\right|^{2}+2v\log v\left<\nabla v, \nabla a \right>\\
&\quad +2a \left|\nabla v\right|^{2}+2v\left<\nabla v, \nabla b' \right>.
\end{aligned}
\end{equation}
\end{lemma}
\begin{theorem}\label{thm;3}
Let $(M, g)$ be a Riemannian manifold of dimension $n\geqslant 2$ with $\mathrm{Ric}\,_{M}\geqslant -k$ for some $k\geqslant 0$. Suppose that $u$ is any positive solution to (\ref{equ;1}) on $Q_{R}=B_{R}\times (t_{0}-R^{2}, t_{0}]$,  where $1-\frac{2}{n}<m<1$. Suppose also that $u$ is bounded from below and above by positive constants. Assume that for $0<R\leqslant 1$, $a$, $b$, $\left|\nabla a\right|$ and $\left|\nabla b\right|$ are bounded functions on $Q_{R}$, and denote
\begin{equation*}
\sup_{Q_{R}}\left|a\right|=D_{1}\quad ,\quad \sup_{Q_{R}}\left|b\right|=D_{2}\quad ,\quad \sup_{Q_{R}}\left|\nabla a\right|=D_{3}\quad ,\quad \sup_{Q_{R}}\left|\nabla b\right|=D_{4}. 
\end{equation*}
Let $v=m u^{m-1}/(1-m)$, then we have
\begin{equation}\label{equ;47}
\sup_{(x, t)\in Q_{R/2}}\left( v^{-\frac{2-m}{4(1-m)}}\left|\nabla v\right|+v^{-\frac{2-m}{4(1-m)}} \right) \leqslant C \frac{1+\sqrt{k}R}{R}, 
\end{equation}
where $C=C(m, n, D_{1}, D_{2}, D_{3}, D_{4}, v_{\max}, v_{\min})$ is a constant. 

\begin{proof}
    We only prove for $n\geqslant 3$. The idea of the proof is the same as that of Theorem \ref{thm;R1}, so that we may omit some details. We write $b'$ as $b$ for simplicity. For any $\displaystyle{\frac{R}{2}\leqslant \rho \leqslant R}$ and $t_{0}-\rho ^{2}<\tau \leqslant t_{0}$, we denote
    \begin{equation*}
    Q_{\rho }=B_{\rho }\times (t_{0}-\rho ^{2}, t_{0}] \quad \text{and}\quad Q_{\rho , \tau }=B_{\rho }\times (t_{0}-\rho ^{2}, \tau ]. 
    \end{equation*}
    Let
    \begin{equation*}
    f=\left|\nabla v\right|^{2}+1 \quad \text{and}\quad F=v^{-\beta }f=\frac{\left|\nabla v\right|^{2}+1}{v^{\beta }}, 
    \end{equation*}
    where $\beta $ is a (positive) constant to be determined later. Applying Lemma \ref{lemma;5} we have
    \begin{align*}
\mathcal{L}(F)&=f\mathcal{L}(v^{-\beta })+v^{-\beta }\mathcal{L}(f)-2(1-m)v\left<\nabla v^{-\beta },\nabla f \right> \\*
&=\beta (1-(1-m)(\beta +1))v^{-\beta -1}f(f-1)-\beta v^{-\beta }(a\log v+b)f \\
&\quad -2(1-m)v^{-\beta +1}\left|\nabla ^{2}v\right|^{2}-2(1-m)v^{-\beta +1}\mathrm{Ric}\,(\nabla v, \nabla v) \\
&\quad +2(1-m)v^{-\beta }(f-1)\Delta v-2v^{-\beta }\left<\nabla v, \nabla f \right>+2(a\log v+b)v^{-\beta }(f-1) \\
&\quad +2v^{-\beta +1}\log v\left<\nabla v,  \nabla a \right>+2av^{-\beta }(f-1)+2v^{-\beta +1}\left<\nabla v, \nabla b \right> \\*
&\quad +2(1-m)\beta v^{-\beta }\left<\nabla v, \nabla f \right>.
\end{align*}
    Hence
    \begin{align*}
-\mathcal{L}(F)&\geqslant \frac{2(1-m)}{n}v^{-\beta +1}(\Delta v)^{2}-2(1-m)v^{-\beta }(f-1)\Delta v-2(1-m)kv^{-\beta +1}(f-1) \\*
&\quad -\beta (1-(1-m)(\beta +1))v^{-\beta -1}f^{2}+2(1-\beta (1-m))v^{-\beta }\left<\nabla v, \nabla f \right> \\
&\quad -(2-\beta )(a\log v+b)v^{-\beta }f-2v^{-\beta +1}\log v\left<\nabla v, \nabla a \right>-2av^{-\beta }f+2av^{-\beta } \\
&\quad -2v^{-\beta +1}\left<\nabla v, \nabla b \right>+\beta (1-(1-m)(\beta +1))v^{-\beta -1}f+2(a\log v+b)v^{-\beta } \\
&\geqslant \left\{ -(1-m)\beta ^{2}+(2-m)\beta -\frac{n(1-m)}{2} \right\} v^{-\beta -1}f^{2}+n(1-m)v^{-\beta -1}f \\
&\quad -\frac{n(1-m)}{2}v^{-\beta -1}+2(1-\beta (1-m))\left<\nabla F, \nabla v \right>-2(1-m)kv^{-\beta +1}f \\
&\quad +2(1-m)kv^{-\beta +1}-2\beta (1-\beta (1-m))v^{-\beta -1}f-C_{1}v^{-\beta }(f+\left|\nabla a\right|^{2}) \\
&\quad -2\left|a\right|v^{-\beta }f-v^{-\beta +1}f-\left|\nabla b\right|^{2}v^{-\beta +1}-(2+\beta )\left( \frac{C_{1}\left|a\right|}{v}+\left|b\right| \right) v^{-\beta }f \\
&\quad +\beta (1-(1-m)(\beta +1))v^{-\beta -1}f-2\left( \frac{C_{1}\left|a\right|}{v}+\left|b\right| \right) v^{-\beta }-2\left|a\right|v^{-\beta } \\
&\geqslant \left\{ -(1-m)\beta ^{2}+(2-m)\beta -\frac{n(1-m)}{2} \right\} v^{-\beta -1}f^{2} \\
&\quad -\left\{ -(1-m)\beta ^{2}+(2-m)\beta -\frac{n(1-m)}{2}+(4+\beta )C_{1}\left|a\right| \right\} v^{-\beta -1}f \\
&\quad -\left\{ C_{1}(1+\left|\nabla a\right|^{2})+4\left|a\right|+v+\left|\nabla b\right|^{2}v+(4+\beta )\left|b\right| \right\} v^{-\beta }f \\*
&\quad +2(1-\beta (1-m))\left<\nabla F, \nabla v \right>-2(1-m)kv^{-\beta +1}f,
\end{align*}
    where $C_{1}=\max\left\{ e^{-1}, v_{\max}\left|\log v_{\max}\right| \right\} $ and we take $\displaystyle{\beta =\frac{1}{2}\frac{2-m}{1-m}}>0$. One readily checks that
    \begin{equation*}
    A=-(1-m)\beta ^{2}+(2-m)\beta -\frac{n(1-m)}{2}>0
    \end{equation*}
    when $\displaystyle{1-\frac{2}{n}<m<1}$. Hence we conclude that 
    \begin{equation}\label{equ;48}
    \mathcal{L}(F)+Av^{-\beta -1}f^{2}\leqslant C_{2}v^{-\beta -1}f-2(1-\beta (1-m))\left<\nabla F, \nabla v \right>+2(1-m)kv^{-\beta +1}f, 
    \end{equation}
    where $C_{2}=C_{2}(m, n, \beta , D_{1}, D_{2}, D_{3}, D_{4}, v_{\max})$ and $C_{2}\geqslant A$. 
    
    Next, let $\zeta $ be a cut-off function supported on $Q_{\rho }$. For any $l\geqslant 0$ we multiply by $\zeta ^{2}F^{l}$ on both sides of (\ref{equ;48}) and integrate on $Q_{\rho , \tau }$. Integrating by parts implies
    \begin{equation}\label{equ;49}
    \begin{aligned}
    &\frac{1}{l+1}\int_{B_{\rho }}\zeta ^{2}F^{l+1}(x, \tau )+(1-m)l\int \zeta ^{2}vF^{l-1}\left|\nabla F\right|^{2}+A\int \zeta ^{2}v^{\beta -1}F^{l+2}\\
    &\quad \leqslant \frac{1}{l+1}\int \frac{\partial \zeta ^{2}}{\partial t}F^{l+1}+C_{3}\int \zeta ^{2}v^{-1}F^{l+1}-\int \zeta ^{2}F^{l}\left<\nabla F, \nabla v \right>\\
    &\quad \quad -2(1-m)\int \zeta vF^{l}\left<\nabla F, \nabla \zeta  \right>+2(1-m)k\int \zeta ^{2}vF^{l+1}.
    \end{aligned}
    \end{equation}
    This inequality is almost the same as (\ref{equ;14}) formally. One can repeat the remaining argument and derive (we must emphasize that the detail may be slightly different since $\beta $ is now greater than $0$, which makes some $v_{\max}$ become $v_{\min}$ or conversely)
    \begin{equation}\label{equ;50}
    \sup_{(x, t)\in Q_{R/2}}F(x, t)\leqslant C_{3} \left( \frac{e^{l_{0}}}{R^{4/n}\left|B_{R}\right|^{2/n}} \right) ^{\frac{n^{2}}{2(n+2)(l_{0}+1)}}\left\|F\right\| _{L^{\gamma_{1}}(Q_{3R/4})}, 
    \end{equation}
    and
    \begin{equation}\label{equ;51}
    \left\|F\right\|_{L^{\gamma_{1}}(Q_{3R/4})}\leqslant C_{4}\left( e^{l_{0}}R^{2}\left|B_{R}\right| \right) ^{\frac{n}{(n+2)(l_{0}+1)}}\frac{l_{0}^{2}}{R^{2}}, 
    \end{equation}
    where $C_{3}$ and $C_{4}$ depend on $m, n, D_{1}, D_{2}, D_{3}, D_{4}, v_{\max}, v_{\min}$. $l_{0}$ is chosen by
    \begin{equation*}
    l_{0}=\delta_{0}(1+\sqrt{k}R)\quad ,\quad \delta_{0}=\delta_{0}(m, n)=\max\left\{ C_{n}+2, \frac{2}{(1-m)A}, n \right\} . 
    \end{equation*}
    Finally, we combine (\ref{equ;51}) with (\ref{equ;50}) and derive
    \begin{equation*}
    \sup_{(x, t)\in Q_{R/2}}F(x, t)\leqslant C_{5} \left( \frac{1+\sqrt{k}R}{R} \right) ^{2}. 
    \end{equation*}
    Since $\displaystyle{F=v^{-\beta }(\left|\nabla v\right|^{2}+1)=v^{-\frac{2-m}{2(1-m)}}\left|\nabla v\right|^{2}+v^{-\frac{2-m}{2(1-m)}}}$, this completes the proof. 
\end{proof}
\end{theorem}

\allowdisplaybreaks[4]
\raggedbottom
\setlength{\emergencystretch}{2em}
\begin{theorem}\label{thm;4}
Let $(M,g)$ be a complete Riemannian manifold of dimension $n\geqslant2$ without boundary and with $\mathrm{Ric}\,_{M}\geqslant-k$ for some $k\geqslant0$. Suppose that $u$ is any positive smooth solution to (\ref{equ;1}) on $Q_R=B_R\times(t_0-R^2,t_0]$, where $1-1/\sqrt{n-1}<m<1$ and $a,b$ are constants. Suppose also that $0<u\leqslant M$. When $m<1/2$, assume in addition that $u\geqslant u_*>0$. Let $v=mu^{m-1}/(1-m)$. If
\begin{equation}\label{equ;52}
 (2m-1)(a\log u+b)+2a\leqslant0\qquad\text{on }Q_R,
\end{equation}
or equivalently, one of the following conditions holds,
\begin{alignat*}{2}
m>\frac12,&\quad a>0, &\hspace{10pt} a\log u_{\max}+\frac{2a}{2m-1}+b&\leqslant0, \\*
m>\frac12,&\quad a=0, &\hspace{10pt} b&\leqslant0, \\
m>\frac12,&\quad a<0, &\hspace{10pt} a\log u_{\min}+\frac{2a}{2m-1}+b&\leqslant0, \\
m=\frac12,&\quad a\leqslant0,&\hspace{10pt}& \\
m<\frac12,&\quad a>0, &\hspace{10pt} a\log u_{\min}+\frac{2a}{2m-1}+b&\geqslant0, \\
m<\frac12,&\quad a=0, &\hspace{10pt} b&\geqslant0, \\*
m<\frac12,&\quad a<0, &\hspace{10pt} a\log u_{\max}+\frac{2a}{2m-1}+b&\geqslant0,
\end{alignat*}
Then we have
\begin{equation}\label{equ;53}
 \sup_{(x,t)\in Q_{R/2}}v^{-\frac{1}{2(1-m)}}|\nabla v|
 \leqslant C\frac{1+\sqrt{k}R}{R}.
\end{equation}
Here $u_{\min}=\inf_{Q_R}u$, $u_{\max}=\sup_{Q_R}u$; conditions involving $\log u_{\min}$ include $u_{\min}>0$. The constant is $C=C(m,n,M)$ when $m\geqslant1/2$, whereas $C=C(m,n,M,u_*)$ when $m<1/2$. In particular, no independent positive lower bound for $u$ is required when $m\geqslant1/2$.
\begin{remark}
When $n\geqslant5$, $m>1-1/\sqrt{n-1}\geqslant1/2$.
\end{remark}
\begin{proof}
Let  $v\geqslant v_0:=\frac{m}{1-m}M^{m-1}>0$, and  
\begin{equation*}
 f=|\nabla v|^2,\qquad F=v^{-\beta}f=\frac{|\nabla v|^2}{v^\beta},
\end{equation*}
where $\beta$ is a positive constant to be determined. By Lemma \ref{lemma;5},
\begin{equation*}
\begin{aligned}
 \mathcal L(F)
 &=\beta[1-(1-m)(\beta+1)]v^{-\beta-1}f^2
   -\beta(a\log v+b')v^{-\beta}f-2(1-m)v^{1-\beta}|\nabla^2v|^2\\
   &\quad-2(1-m)v^{1-\beta}\mathrm{Ric}\,(\nabla v,\nabla v)+2(1-m)v^{-\beta}f\Delta v-2v^{-\beta}\langle\nabla v,\nabla f\rangle\\
 &\quad+2(a\log v+b')v^{-\beta}f+2av^{-\beta}f
   +2(1-m)\beta v^{-\beta}\langle\nabla v,\nabla f\rangle.
\end{aligned}
\end{equation*}
For a real number $\varepsilon$ to be determined later, direct computation yields
\begin{align*}
-\mathcal L(F)+\varepsilon\langle\nabla F,\nabla v\rangle
 &=\beta[-1+(1-m)(\beta+1)-\varepsilon]v^{-\beta-1}f^2 \\*
&\quad-[(2-\beta)(a\log v+b')+2a]v^{-\beta}f \\
&\quad+2(1-m)v^{1-\beta}v_{ij}^2
       -2(1-m)v^{-\beta}fv_{jj} \\
&\quad+[2\varepsilon+4(1-\beta(1-m))]v^{-\beta}v_iv_jv_{ij} \\*
&\quad+2(1-m)v^{1-\beta}\mathrm{Ric}\,(\nabla v,\nabla v).
\end{align*}
At points where $f>0$, put $A=(v_{ij})$ and $e=\nabla v/|\nabla v|$. Completing the square and applying Lemma \ref{lemma;4}, we obtain
\begin{align}
&-\mathcal L(F)+\varepsilon\langle\nabla F,\nabla v\rangle
       +[(2-\beta)(a\log v+b')+2a]v^{-\beta}f \label{equ;54} \\*
&\quad\geqslant-\frac{v^{-\beta-1}f^2}{2(1-m)}
   \left[(\varepsilon+2(1-\beta(1-m)))\frac{A(e,e)}{|A|}
          -(1-m)\frac{\mathrm{tr}\,A}{|A|}\right]^2 \notag \\
&\qquad\quad+\beta[-1+(1-m)(\beta+1)-\varepsilon]v^{-\beta-1}f^2
            -2(1-m)kv^{1-\beta}f \notag \\
&\quad\geqslant-\frac{1}{2(1-m)}
  \bigl\{2(1-m)^2\beta^2-2(1-m)(m+2+\varepsilon)\beta \notag \\*
&\hspace{30mm}+(m+1+\varepsilon)^2+(n-1)(1-m)^2\bigr\}
        v^{-\beta-1}f^2-2(1-m)kv^{1-\beta}f. \notag
\end{align}
Take
\begin{equation*}
 \beta =\frac{m+2+\varepsilon }{2(1-m)}\left( =\frac{1}{1-m} \right),\qquad \varepsilon=-m,\qquad
 A=\frac{1-(n-1)(1-m)^2}{2(1-m)}>0.
\end{equation*}
Then (\ref{equ;54}) becomes
\begin{equation}\label{equ;55}
\begin{aligned}
 \mathcal L(F)+Av^{\beta-1}F^2
 &\leqslant-m\langle\nabla F,\nabla v\rangle+2(1-m)kvF\\
 &\quad+[(2-\beta)(a\log v+b')+2a]F.
\end{aligned}
\end{equation}
The last coefficient equals the left-hand side of (\ref{equ;52}). We retain it until after the weighted time derivative has been computed.

Choose $0\leqslant\zeta\leqslant1$, equal to one on $Q_{R/2}$, supported in $B_{3R/4}\times(t_0-R^2,t_0]$, and zero near $t_0-R^2$, such that
\begin{equation}\label{equ;82}
 |\zeta_t|\leqslant\frac{C(n)}{R^2},\qquad
 \frac{|\nabla\zeta|^2}{\zeta}\leqslant\frac{C(n)}{R^2},\qquad
 -\Delta\zeta\leqslant C(n)\left(\frac1{R^2}+\frac{\sqrt{k}}R\right).
\end{equation}
Such a cut-off function is obtained from a squared nonincreasing distance cutoff and a time cutoff by Laplacian comparison. The Laplacian inequality is understood distributionally at the cut locus. Set
\begin{equation}\label{equ;83}
 H=\zeta F.
\end{equation}
On $\{\zeta>0\}$, the product rule yields
\begin{equation}\label{equ;84}
\begin{aligned}
 \mathcal L(H)+A\frac{v^{\beta-1}}{\zeta}H^2
 &\leqslant-m\left<\nabla v,\nabla H\right>-\frac{2(1-m)v}{\zeta}\left<\nabla\zeta,\nabla H\right>\\
 &\quad+\left[2(1-m)kv+\frac{\zeta_t}{\zeta}
             -(1-m)v\frac{\Delta\zeta}{\zeta}
             +2(1-m)v\frac{|\nabla\zeta|^2}{\zeta^2}\right]H\\
 &\quad+m\frac{\langle\nabla\zeta,\nabla v\rangle}{\zeta}H
       +[(2-\beta)(a\log v+b')+2a]H.
\end{aligned}
\end{equation}
Let $\lambda>0$ and $p=l+1>2$. Multiply (\ref{equ;84}) by $\zeta v^{-\lambda}H^l$ and integrate on $Q_{R,\tau}=B_R\times(t_0-R^2,\tau]$. The time term is
\begin{align*}
\int\zeta v^{-\lambda}H^lH_t
 &=\frac1p\int_{B_R}\zeta v^{-\lambda}H^p(x,\tau)
   -\frac1p\int\zeta_t v^{-\lambda}H^p
   +\frac\lambda p\int\zeta v^{-\lambda-1}v_tH^p, \\*
\frac\lambda p\int\zeta v^{-\lambda-1}v_tH^p
 &=-\frac{\lambda(1-m)}p\int v^{-\lambda}H^p
                         \langle\nabla\zeta,\nabla v\rangle \\
&\quad-\lambda(1-m)\int\zeta v^{-\lambda}H^{p-1}
                         \langle\nabla H,\nabla v\rangle \\
&\quad+\frac{\lambda[(1-m)\lambda-1]}p
          \int v^{\beta-\lambda-1}H^{p+1} \\*
&\quad+\frac\lambda p\int\zeta(a\log v+b')v^{-\lambda}H^p.
\end{align*}
Here we used $\zeta|\nabla v|^2=v^\beta H$ and
\begin{equation*}
\begin{aligned}
 \int\zeta v^{-\lambda}H^p\Delta v
 &=-\int v^{-\lambda}H^p\langle\nabla\zeta,\nabla v\rangle
    +\lambda\int v^{\beta-\lambda-1}H^{p+1}\\
 &\quad-p\int\zeta v^{-\lambda}H^{p-1}\langle\nabla H,\nabla v\rangle.
\end{aligned}
\end{equation*}
Integrating the Laplacian term by parts implies
\begin{equation*}
\begin{aligned}
 &-(1-m)\int\zeta v^{1-\lambda}H^l\Delta H\\
 &\quad=(1-m)l\int\zeta v^{1-\lambda}H^{p-2}|\nabla H|^2
       +(1-m)\int v^{1-\lambda}H^{p-1}\langle\nabla\zeta,\nabla H\rangle\\
 &\qquad\quad+(1-m)(1-\lambda)\int\zeta v^{-\lambda}H^{p-1}
                                      \langle\nabla v,\nabla H\rangle.
\end{aligned}
\end{equation*}
Combining these identities, we deduce
\begin{align}
&\frac1p\int_{B_R}\zeta v^{-\lambda}H^p(x,\tau)
 +(1-m)l\int\zeta v^{1-\lambda}H^{p-2}|\nabla H|^2 \label{equ;56} \\*
&\quad+\left(A+\frac{\lambda[(1-m)\lambda-1]}p\right)
                  \int v^{\beta-\lambda-1}H^{p+1} \notag \\
&\leqslant[2(1-m)\lambda-1]\int\zeta v^{-\lambda}H^{p-1}
                                  \langle\nabla H,\nabla v\rangle \notag \\
&\quad-3(1-m)\int v^{1-\lambda}H^{p-1}\langle\nabla H,\nabla\zeta\rangle
    +\left(m+\frac{(1-m)\lambda}p\right)
             \int v^{-\lambda}H^p\langle\nabla\zeta,\nabla v\rangle \notag \\
&\quad+\int\left[\left(1+\frac1p\right)\zeta_t v^{-\lambda}
                  -(1-m)v^{1-\lambda}\Delta\zeta\right]H^p \notag \\
&\quad+2(1-m)\int v^{1-\lambda}
                     \left(\frac{|\nabla\zeta|^2}{\zeta}+k\zeta\right)H^p \notag \\*
&\quad+\int\zeta v^{-\lambda}
       \left[\left(2-\beta-\frac\lambda p\right)(a\log v+b')+2a\right]H^p. \notag
\end{align}
Now choose
\begin{equation}\label{equ;85}
 \lambda=\beta-1=\frac{m}{1-m}.
\end{equation}
Then $v^{\beta-\lambda-1}=1$, $\lambda[(1-m)\lambda-1]=-m$. By Cauchy--Schwarz, 
\begin{align}
&(2m-1)\int\zeta v^{-\lambda}H^{p-1}
                    \langle\nabla H,\nabla v\rangle
 \leqslant\frac{(1-m)l}{4}\int\zeta v^{1-\lambda}H^{p-2}|\nabla H|^2
          +\frac{(2m-1)^2}{(1-m)l}\int H^{p+1}, \label{equ;86} \\*[4pt]
&-3(1-m)\int v^{1-\lambda}H^{p-1}
                    \langle\nabla H,\nabla\zeta\rangle
 \leqslant\frac{(1-m)l}{4}\int\zeta v^{1-\lambda}H^{p-2}|\nabla H|^2+\frac{9(1-m)}l\int v^{1-\lambda}\frac{|\nabla\zeta|^2}{\zeta}H^p, \notag \\*[4pt]
&m\left(1+\frac1p\right)
      \int v^{-\lambda}H^p\langle\nabla\zeta,\nabla v\rangle
 \leqslant\frac 14  A\int H^{p+1}
          +\frac{m^2(1+1/p)^2}{A}\int v^{1-\lambda}\frac{|\nabla\zeta|^2}{\zeta}H^p. \notag
\end{align}
Set
\begin{equation*}
 n'=\begin{cases}n,&n\geqslant3,\\4,&n=2,\end{cases}
 \qquad \chi=\frac{n'+2}{n'},
\end{equation*}
and choose
\begin{equation}\label{equ;87}
 l_0=\max\left\{2,\frac{n'}2,\frac{8m}{A},
            \frac{8(2m-1)^2}{(1-m)A}\right\}.
\end{equation}
For $l\geqslant l_0$, the remaining highest-order coefficient is
\begin{equation*}
 \frac{3A}{4}-\frac{m}{p}-\frac{(2m-1)^2}{(1-m)l}\geqslant\frac A2.
\end{equation*}
By (\ref{equ;52}), the full reaction term is bounded by
\begin{equation}\label{equ;88}
\begin{aligned}
 &\zeta v^{-\lambda}
       \left[\left(2-\beta-\frac\lambda p\right)(a\log v+b')+2a\right]H^p
       \leqslant\frac Lp H^p,\\
 &L=\lambda\sup_{s\geqslant v_0}\frac{|a\log s+b'|}{s^\lambda}<\infty.
\end{aligned}
\end{equation}
Let
\begin{equation*}
 P_\beta=\begin{cases}
 v_0^{2-\beta},&m\geqslant1/2,\\
 V^{2-\beta},&m<1/2,\quad V=\dfrac{m}{1-m}u_*^{m-1}.
 \end{cases}
\end{equation*}
Then $v^{-\lambda}\leqslant v_0^{1-\beta}$ and $v^{1-\lambda}\leqslant P_\beta$. It follows from (\ref{equ;82}), (\ref{equ;56}) and (\ref{equ;86})--(\ref{equ;88}) that
\begin{equation}\label{equ;89}
 \frac1p\int_{B_R}\zeta v^{-\lambda}H^p(x,\tau)
 +\frac{(1-m)l}{2}\int\zeta v^{1-\lambda}H^{p-2}|\nabla H|^2+\frac A2\int H^{p+1}
 \leqslant\left(K+\frac Lp\right)\int H^p,
\end{equation}
where, with $C_1=C_1(m,n)$ sufficiently large,
\begin{equation}\label{equ;90}
 K=C_1\left\{\frac{v_0^{1-\beta}}{R^2}
                  +P_\beta\left(\frac1{R^2}+k\right)\right\}>0.
\end{equation}
We have used $\sqrt{k}/R\leqslant(R^{-2}+k)/2$.

We next estimate the initial integral and perform the power iteration. Denote
\begin{equation*}
 J_q=\int_{Q_R}H^q\,d\mu dt,\qquad
 T(p)=\frac2A\left(K+\frac Lp\right),\qquad p_0=l_0+2.
\end{equation*}
Taking $\tau=t_0$ in (\ref{equ;89})  and omitting the two
nonnegative terms on the left yields
\begin{equation}\label{equ;92}
 J_{p+1}\leqslant T(p)J_p\qquad(p\geqslant l_0+1).
\end{equation}
The integrals are finite since the support is contained in a compact subcylinder on which the solution is smooth and positive. H\"older's inequality gives
\begin{equation*}
 J_{p_0}\leqslant T(p_0-1)J_{p_0-1}
 \leqslant T(p_0-1)|Q_R|^{1/p_0}J_{p_0}^{(p_0-1)/p_0}.
\end{equation*}
If $J_{p_0}=0$, the result is immediate. Otherwise,
\begin{equation}\label{equ;93}
 J_{p_0}^{1/p_0}\leqslant T(p_0-1)|Q_R|^{1/p_0}.
\end{equation}
Set $p_j=p_0\chi^j$ and $Y_j=J_{p_j}^{1/p_j}$. Since $p_{j+1}-p_j\geqslant1$, H\"older's inequality with $\theta_j=(p_{j+1}-p_j)^{-1}$ gives
\begin{equation*}
 J_{p_{j+1}-1}\leqslant J_{p_j}^{\theta_j}J_{p_{j+1}}^{1-\theta_j}.
\end{equation*}
Combining this with (\ref{equ;92}) at $p=p_{j+1}-1$, we obtain
\begin{equation}\label{equ;94}
\begin{aligned}
 J_{p_{j+1}}&\leqslant T(p_{j+1}-1)^{p_{j+1}-p_j}J_{p_j},\\
 Y_{j+1}&\leqslant T(p_{j+1}-1)^{1-1/\chi}Y_j^{1/\chi}.
\end{aligned}
\end{equation}
Consequently,
\begin{equation}\label{equ;95}
 Y_j\leqslant Y_0^{\chi^{-j}}
     \prod_{i=0}^{j-1}T(p_{i+1}-1)^{(1-1/\chi)\chi^{-(j-1-i)}}.
\end{equation}
The exponents in the product sum to $1-\chi^{-j}$ and $T(p_{i+1}-1)\to2K/A$. For every $\epsilon>0$, the factors after a fixed index are at most $2K/A+\epsilon$; the total exponent of the finitely many earlier factors tends to zero. Thus
\begin{equation*}
 \limsup_{j\to\infty}Y_j\leqslant\frac{2K}{A}.
\end{equation*}
On the fixed finite cylinder $Q_R$, $\lim_{j\to\infty}J_{p_j}^{1/p_j}=\sup_{Q_R}H$. Since $\zeta=1$ on $Q_{R/2}$, it follows that
\begin{equation}\label{equ;96}
 \sup_{Q_{R/2}}F\leqslant\frac{2K}{A}
 \leqslant C(m,n)\left\{\frac{v_0^{1-\beta}}{R^2}
              +P_\beta\left(\frac1{R^2}+k\right)\right\}.
\end{equation}
Taking square roots proves (\ref{equ;53}).
\end{proof}
\end{theorem}

By a slight modification of the proof of Theorem \ref{thm;4}, we can derive another Hamilton type gradient estimate without a condition of the form (\ref{equ;52}).

\begin{theorem}\label{thm;5}
Let $(M,g)$ be a complete Riemannian manifold of dimension $n\geqslant2$ without boundary and with $\mathrm{Ric}\,_{M}\geqslant-k$ for some $k\geqslant0$. Suppose that $u$ is any positive smooth solution to (\ref{equ;1}) on $Q_R=B_R\times(t_0-R^2,t_0]$, where $(n-1)/(n+3)<m<1$ and $a,b$ are constants, $a\leqslant0$. Suppose also that $0<u\leqslant M$. Let $v=mu^{m-1}/(1-m)$. Then we have
\begin{equation}\label{equ;63}
 \sup_{(x,t)\in Q_{R/2}}\frac{|\nabla v|}{v}
 \leqslant C\frac{1+\sqrt{k}R}{R},
\end{equation}
where $C=C(m,n,M)$ is a constant, independent of $\inf_{Q_R}u$ and of $a,b$.
\begin{proof}
Let $v_0=\frac{m}{1-m}M^{m-1}>0$, $f=|\nabla v|^2$ and $F=v^{-\beta}f$. As in the proof of Theorem \ref{thm;4}, we have
\begin{equation}\label{equ;64}
\begin{aligned}
 &-\mathcal L(F)+\varepsilon\langle\nabla F,\nabla v\rangle
       +[(2-\beta)(a\log v+b')+2a]v^{-\beta}f\\
 &\quad\geqslant-\frac{1}{2(1-m)}
   \bigl\{2(1-m)^2\beta^2-2(1-m)(m+2+\varepsilon)\beta\\
 &\hspace{44mm}+(m+1+\varepsilon)^2+(n-1)(1-m)^2\bigr\}
          v^{-\beta-1}f^2-2(1-m)kv^{1-\beta}f.
\end{aligned}
\end{equation}
Take $\beta=2$ and $\varepsilon=1-3m$. Then
\begin{equation}\label{equ;65}
 \mathcal L(F)+AvF^2
 \leqslant(1-3m)\langle\nabla F,\nabla v\rangle
             +2(1-m)kvF+2aF,
\end{equation}
where $A=2-(n+3)(1-m)/2>0$ by assumption. We still keep the term $2aF$.

Choose $\zeta$ with the properties in (\ref{equ;82}), and $H=\zeta F$. For a general $\lambda>0$, multiply the differential inequality for $H$ by $\zeta v^{-\lambda}H^l$, with $p=l+1$. The highest-order integral has weight $v^{1-\lambda}$. We therefore take $\lambda=1$, obtaining
\begin{align}
&\frac1p\int_{B_R}\frac{\zeta}{v}H^p(x,\tau)
 +(1-m)l\int\zeta H^{p-2}|\nabla H|^2
 +\left(A-\frac mp\right)\int H^{p+1} \label{equ;97} \\*
&\leqslant(2-4m)\int\frac{\zeta}{v}H^{p-1}
                          \langle\nabla H,\nabla v\rangle
       -3(1-m)\int H^{p-1}\langle\nabla H,\nabla\zeta\rangle \notag \\
&\quad+\left(3m-1+\frac{1-m}{p}\right)
              \int\frac{H^p}{v}\langle\nabla\zeta,\nabla v\rangle \notag \\
&\quad+\int\left[\left(1+\frac1p\right)\frac{\zeta_t}{v}
                         -(1-m)\Delta\zeta\right]H^p \notag \\*
&\quad+2(1-m)\int\left(\frac{|\nabla\zeta|^2}{\zeta}+k\zeta\right)H^p
       +\int\frac{\zeta}{v}\left[2a-\frac{a\log v+b'}p\right]H^p. \notag
\end{align}
The assumptions $a\leqslant0$ and $v\geqslant v_0$ imply
\begin{equation}\label{equ;98}
 \frac{\zeta}{v}\left[2a-\frac{a\log v+b'}p\right]H^p
 \leqslant\frac Lp H^p,
 \qquad L=\sup_{s\geqslant v_0}\frac{|a\log s+b'|}{s}<\infty.
\end{equation}
In particular, no sign of $a\log v+b'$ is required. Choose
\begin{equation*}
 l_0=\max\left\{2,\frac{n'}2,\frac{8m}{A},
                  \frac{8(2-4m)^2}{(1-m)A}\right\},
 \qquad p_0=l_0+2,
\end{equation*}
where $n'=n$ for $n\geqslant3$ and $n'=4$ for $n=2$.
For $l\geqslant l_0$ the coefficient
\begin{equation*}
 \frac{3A}{4}-\frac mp-\frac{(2-4m)^2}{(1-m)l}
 \geqslant\frac A2.
\end{equation*}
Following the argument as in the proof in \ref{thm;4}, and using (\ref{equ;82}) and (\ref{equ;98}), we obtain
\begin{equation*}
 \frac1p\int_{B_R}\frac{\zeta}{v}H^p(x,\tau)
 +\frac{(1-m)l}{2}\int\zeta H^{p-2}|\nabla H|^2+\frac A2\int H^{p+1}
 \leqslant\left(K+\frac Lp\right)\int H^p,
\end{equation*}
where
\begin{equation}\label{equ;99}
 K=C_1(m,n)\left(\frac{v_0^{-1}}{R^2}+\frac1{R^2}+k\right)>0.
\end{equation}
Then the standard iteration implies
\begin{equation*}
 \sup_{Q_{R/2}}\frac{|\nabla v|^2}{v^2}
 \leqslant\sup_{Q_R}H
 \leqslant\frac{2K}{A}
 \leqslant C(m,n)\left(\frac{v_0^{-1}}{R^2}+\frac1{R^2}+k\right).
\end{equation*}
Taking square roots proves (\ref{equ;63}).
\end{proof}
\end{theorem}

Now we can derive a Liouville type theorem for FDE type equations from Theorem \ref{thm;4} and Theorem \ref{thm;5}.

\begin{proof}[Proof of Corollary \ref{cor;R4}]
We first clarify the relationship between $1-\frac{1}{\sqrt{n-1}}$, $\frac{n-1}{n+3}$ and $\frac12$. Direct computation yields
\begin{equation*}
 \frac{n-1}{n+3}-\left(1-\frac1{\sqrt{n-1}}\right)
 =\frac1{\sqrt{n-1}}-\frac4{n+3}
 =\frac{(\sqrt{n-1}-2)^2}{(n+3)\sqrt{n-1}}\geqslant0,
\end{equation*}
where equality holds if and only if $n=5$. Hence
\begin{equation}\label{equ;66}
\begin{cases}
 \dfrac{n-1}{n+3}>1-\dfrac1{\sqrt{n-1}}>\dfrac12,&n>5,\\[1mm]
 \dfrac{n-1}{n+3}=1-\dfrac1{\sqrt{n-1}}=\dfrac12,&n=5,\\[1mm]
 1-\dfrac1{\sqrt{n-1}}<\dfrac{n-1}{n+3}<\dfrac12,&2\leqslant n\leqslant4.
\end{cases}
\end{equation}

Whenever the hypotheses of Theorem \ref{thm;4} or Theorem \ref{thm;5} hold on $M\times\mathbb{R}$ with the required global bounds, the constant in the gradient estimate is independent of $R$. Taking $k=0$ and letting $R\to\infty$ at any fixed $(x_0,t_0)$, we obtain $\nabla v=0$. Solving the ODE with respect to $t$, we derive
\begin{equation*}
 u(x,t)=
 \begin{cases}
  e^{(C+bt)},&a=0,\\
  e^{(-b/a+Ce^{at})},&a\neq0,
 \end{cases}
\end{equation*}
where $C$ is a constant. If $a=0$, boundedness from above on the whole time axis forces $b=0$. If $a\neq0$, the same boundedness forces $C\leqslant0$. Writing $C=-C_0$, we have $C_0\geqslant0$; when $C_0>0$, the infimum of $U$ on $\mathbb{R}$ is zero and its supremum is $e^{-b/a}$. In particular, a positive lower bound forces $C_0=0$.

By Theorem \ref{thm;4} and this observation, we have
\begin{flushleft}
(1$'$) When $m>1/2$,

$\quad$(i) If $a>0$ and $u\leqslant e^{-2/(2m-1)}e^{-b/a}$, then (\ref{equ;52}) holds without a lower-bound assumption. The resulting spatially constant solution satisfies $U(t)\to e^{-b/a}$ as $t\to-\infty$, contradicting $e^{-2/(2m-1)}e^{-b/a}<e^{-b/a}$. Thus
\begin{equation*}
 \sup_{M\times\mathbb{R}}u>e^{-\frac2{2m-1}}e^{-\frac ba};
\end{equation*}

$\quad$(ii) If $a=0$ and $b<0$, then (\ref{equ;52}) holds, whereas $U(t)=\exp(C+bt)$ is unbounded as $t\to-\infty$. Hence such $u$ does not exist;

$\quad$(iii) If $a=b=0$, then $U'=0$, and $u$ must be a positive constant;

$\quad$(iv) If $a<0$ and $u\geqslant e^{-2/(2m-1)}e^{-b/a}$, then (\ref{equ;52}) holds. This positive lower bound forces $C_0=0$, so $u\equiv e^{-b/a}$.

(2$'$) When $m<1/2$,

$\quad$(i) If $a>0$ and $u\geqslant e^{-2/(2m-1)}e^{-b/a}$, this assumption itself supplies the positive lower bound required by Theorem \ref{thm;4}, and (\ref{equ;52}) holds. But $U(t)\to e^{-b/a}$ as $t\to-\infty$, contradicting $e^{-2/(2m-1)}e^{-b/a}>e^{-b/a}$. Thus
\begin{equation*}
 \inf_{M\times\mathbb{R}}u<e^{-\frac2{2m-1}}e^{-\frac ba}.
\end{equation*}
For (ii)--(iv) below, assume that $u$ is bounded from below by a positive constant when applying Theorem \ref{thm;4}. This is part of the hypotheses in the smaller range specified in the statement; the remaining range is also covered by (4$'$).

$\quad$(ii) If $a=0$ and $b>0$, then (\ref{equ;52}) holds, whereas $U(t)=\exp(C+bt)$ is unbounded as $t\to+\infty$. Hence such $u$ does not exist;

$\quad$(iii) If $a=b=0$, then $u$ must be a positive constant;

$\quad$(iv) If $a<0$ and $u\leqslant e^{-2/(2m-1)}e^{-b/a}$, then (\ref{equ;52}) holds, and the positive lower bound gives $u\equiv e^{-b/a}$.

(3$'$) When $m=1/2$, the dimensional range implies $2\leqslant n\leqslant4$, and (\ref{equ;52}) reduces to $a\leqslant0$.

$\quad$(i) If $a<0$, then $u(x,t)=\exp(-b/a-C_0e^{at})$ for some $C_0\geqslant0$. It is identically $e^{-b/a}$ if it is also bounded from below by a positive constant;

$\quad$(ii) If $a=0$ and $b\neq0$, then such $u$ does not exist;

$\quad$(iii) If $a=b=0$, then $u$ must be a positive constant.

In particular, $n\geqslant5$ implies $m>1-1/\sqrt{n-1}\geqslant1/2$.
\end{flushleft}

Furthermore, as a direct consequence of Theorem \ref{thm;5}, with no positive lower-bound assumption, we have
\begin{flushleft}
(4$'$) When $(n-1)/(n+3)<m<1$,

$\quad$(i) If $a<0$, then $u(x,t)=\exp(-b/a-C_0e^{at})$ for some $C_0\geqslant0$, and $u\equiv e^{-b/a}$ precisely when $u$ is bounded from below by a positive constant;

$\quad$(ii) If $a=0$ and $b\neq0$, then such $u$ does not exist;

$\quad$(iii) If $a=b=0$, then $u$ must be a positive constant.
\end{flushleft}

We conclude (1)--(4) from (1$'$)--(4$'$). For (1), use (1$'$)(iii) or (3$'$)(iii) when $m\geqslant1/2$, and use (4$'$)(iii) when $(n-1)/(n+3)<m<1/2$; the remaining range uses (2$'$)(iii) and the stated positive lower bound. Part (2) follows from (4$'$)(ii), (2$'$)(ii) and (1$'$)(ii), respectively. Part (3) follows from (2$'$)(i) and (1$'$)(i), without an independent positive lower bound. Finally, in part (4), the first, third and fourth cases follow from (4$'$)(i). The second case follows from (2$'$)(iv) and the stated positive lower bound, while the fifth follows from (1$'$)(iv), whose threshold condition itself provides a positive lower bound. This completes the proof.
\end{proof}

\medskip

{\bf Acknowledgements} The first author is supported by NSFC No. 12671074, 12071352, 12271039.

\vspace{.1in}
\textbf{Conflict of Interest} The authors have no conflict of interest to declare.

\vspace{.1in}
\textbf{Data availability} The authors declare no datasets were generated or analysed during the current study.


\vskip24pt


\begin{thebibliography}{1}


    \bibitem{AB} D.~G. Aronson and P. B\'enilan, R\'egularit\'e{} des solutions de l'\'equation des milieux poreux dans ${\bf R}\sp{N}$, C. R. Acad. Sci. Paris S\'er. A-B {\bf 288} (1979), no.~2, {\rm A}103--{\rm A}105.


    \bibitem{GMP}
    G. Grillo, D.~D. Monticelli and F. Punzo, The porous medium equation on noncompact manifolds with nonnegative Ricci curvature: a Green function approach, J. Differential Equations {\bf 430} (2025), Paper No. 113191, 42 pp.


    
    \bibitem{Ha} 
    R.~S. Hamilton, A matrix Harnack estimate for the heat equation, Comm. Anal. Geom. {\bf 1} (1993), no.~1, 113--126.

    
    
    \bibitem{HHW} 
    D. Han, J. He and Y. Wang, Gradient estimates for $\Delta_pu+A|\nabla u|^q+Bu^r+C=0$ on manifolds and applications, J. Functional Analysis, {\bf 290} (2026), 111274.

   \bibitem{HHL} 
    G. Huang, Z, Huang and H.Li, Gradient estimates for the porous medium equations on Riemannian manifolds, J. Geom. Anal. {\bf 23} (2013), 1851--1875.


   \bibitem{HM} 
    G. Huang and B.~Q. Ma, Hamilton's gradient estimates of porous medium and fast diffusion equations, Geom. Dedicata {\bf 188} (2017), 1--16.


    \bibitem{HXZ}
    G. Huang, R.~W. Xu and F. Zeng, Hamilton's gradient estimates and Liouville theorems for porous medium equations, J. Inequal. Appl. {\bf 2016}, Paper No. 37, 7 pp.


    \bibitem{HS} 
    S. Huang and B. Shen, Gradient estimates for porous medium and fast diffusion equations on Riemannian manifolds via Moser iteration, Commun. Pure Appl. Anal. {\bf 24} (2025), no.~7, 1242--1260.



    
    \bibitem{JWZ} 
    C. Jin, Y. Wang and F. Zeng, Cheng-Yau logarithmic gradient estimates for a nonlinear elliptic equation on smooth metric measure spaces, preprint.




    \bibitem{LY} 
    P. Li and S. T. Yau, On the parabolic kernel of the Schrödinger operator, Acta Math. {\bf 156} (1986), no. 3-4, 153--201.


    \bibitem{L} 
    P. Lu, L. Ni, J-L. V\'azquez and C. Villani, Local Aronson-B\'enilan estimates and entropy formulae for porous medium and fast diffusion equations on manifolds, J. Math. Pures Appl. (9) {\bf 91} (2009), no.~1, 1--19.
    

    \bibitem{S-C} 
    L. Saloff-Coste, Uniformly elliptic operators on Riemannian manifolds, J. Differential Geom. {\bf 36} (1992), no.~2, 417--450.
    
    \bibitem{SZ}
    P. Souplet and Q.~S. Zhang, Sharp gradient estimate and Yau's Liouville theorem for the heat equation on noncompact manifolds, Bull. London Math. Soc. {\bf 38} (2006), no.~6, 1045--1053.


    \bibitem{V}
    J.~L. V\'azquez, Fundamental solution and long time behavior of the porous medium equation in hyperbolic space, J. Math. Pures Appl. (9) {\bf 104} (2015), no.~3, 454--484.																						

    \bibitem{WW1} 
    J. Wang and Y.~D. Wang, Gradient estimates for $\Delta u+a(x)u\log u+b(x)u=0$ and its parabolic counterpart under integral Ricci curvature bounds, Comm. Anal. Geom. {\bf 32} (2024), no.~4, 923--975.



    \bibitem{WW2} 
    J. Wang and Y.~D. Wang, Boundedness and gradient estimates for solutions to $\Delta u+a(x)u\log u+b(x)u=0$ on Riemannian manifolds, J. Differential Equations {\bf 402} (2024), 495--517.


    \bibitem{W}
    W. Wang, Harnack inequality, heat kernel bounds and eigenvalue estimates under integral Ricci curvature bounds, J. Differential Equations {\bf 269} (2020), no.~2, 1243--1277.

    
    
    \bibitem{WXZ} 
    W. Wang, R.~L. Xie and P. Zhang, Some gradient estimates and Liouville properties of the fast diffusion equation on Riemannian manifolds, Chinese Ann. Math. Ser. B {\bf 42} (2021), no.~4, 529--550.
    
																								

    \bibitem{X} 
    X. Xu, Gradient estimates for $u_t=\Delta F(u)$ on manifolds and some Liouville-type theorems, J. Differential Equations {\bf 252} (2012), no.~2, 1403--1420.
    
 


    \bibitem{Yau} 
    S. T. Yau, Harmonic functions on complete Riemannian manifolds, Comm. Pure Appl. Math. {\bf 28} (1975), 201--228.

    \bibitem{ZZ}
    Q.~S. Zhang and M. Zhu, Li-Yau gradient bounds on compact manifolds under nearly optimal curvature conditions, J. Funct. Anal. {\bf 275} (2018), no.~2, 478--515.
															

    \bibitem{Z2} 
    X. Zhu, Hamilton's gradient estimates and Liouville theorems for fast diffusion equations on noncompact Riemannian manifolds, Proc. Amer. Math. Soc. {\bf 139} (2011), no.~5, 1637--1644.

   \bibitem{Z1} 
    X. Zhu, Hamilton's gradient estimates and Liouville theorems for porous medium equations on noncompact Riemannian manifolds, J. Math. Anal. Appl. {\bf 402} (2013), no.~1, 201--206.
 										


													\end{thebibliography}
\end{document}